\documentclass[a4paper]{amsart}

\usepackage[utf8]{inputenc}
\usepackage[T1]{fontenc}
\usepackage{amsthm}
\usepackage{xspace}
\usepackage{amssymb,amsmath,stmaryrd}
\usepackage[english]{babel}
\usepackage[shortlabels]{enumitem}

\usepackage[all]{xy}
\usepackage{xcolor}
\usepackage{xr-hyper}
\usepackage{hyperref}

\usepackage{fixme}

\newtheorem{theorem}{Theorem}[section]
\newtheorem{proposition}[theorem]{Proposition}
\newtheorem{lemma}[theorem]{Lemma}
\newtheorem{corollary}[theorem]{Corollary}
\newtheorem{assumption}[theorem]{Assumption}

\theoremstyle{definition}
\newtheorem{definition}[theorem]{Definition}

\newtheorem{example}[theorem]{Example}

\renewcommand{\phi}{\varphi}
\renewcommand{\epsilon}{\varepsilon}
\newcommand{\N}{\mathbb N}

\newcommand{\C}{\mathbb C}

\renewcommand{\setminus}{\backslash}

\newcommand{\setof}[2]{\left\{ #1 \;\middle|\; #2 \right\}}
\newcommand{\ftn}[3]{ #1 \colon #2 \rightarrow #3 }
\newcommand{\reg}{\textnormal{reg}} % regular vertices
\newcommand{\sing}{\textnormal{sing}} % singular vertices
\newcommand{\starhomo}{\mbox{$\sp*$-}ho\-mo\-morphism\xspace}

\newcommand{\stariso}{\mbox{$\sp*$-}iso\-morphism\xspace}
\newcommand{\starisos}{\mbox{$\sp*$-}iso\-morphisms\xspace}

\DeclareMathOperator{\im}{im} % image

\newcommand{\iso}{\mathrm{iso}} % isolated points
\newcommand{\cycle}{\mathrm{cycle}} % cycle points
\newcommand{\sink}{\mathrm{sink}} % sinks

\numberwithin{equation}{section}

\input xy
\xyoption{all}

\title[Diagonal preserving $^*$-isomorphisms of graph $C^*$-algebras]{A dynamical characterization of diagonal preserving $^*$-isomorphisms of graph $C^*$-algebras}
\author{Sara E. Arklint}
\address{Department of Mathematical Sciences, University of Copenhagen, Uni\-versi\-tets\-parken~5, DK-2100 Copenhagen, Denmark}
\email{arklint@math.ku.dk}

\author{S\o{}ren Eilers}
\address{Department of Mathematical Sciences, University of Copenhagen, Uni\-versi\-tets\-parken~5, DK-2100 Copenhagen, Denmark}
\email{eilers@math.ku.dk}

\author{Efren Ruiz}
\address{Department of Mathematics, University of Hawaii, Hilo, 200 W.~Kawili St., Hilo, Hawaii, 96720-4091 USA}
\email{ruize@hawaii.edu}

\date{\today}

\begin{document}

\begin{abstract}
We characterize when there exists a diagonal preserving \stariso between two graph $C^*$-algebras in terms of the dynamics of the boundary path spaces.  In particular, we refine the notion of ``orbit equivalence'' between the boundary path spaces of the directed graphs $E$ and $F$ and show that this is a necessary and sufficient condition for the existence of a diagonal preserving \stariso between the graph $C^*$-algebras $C^*(E)$ and $C^*(F)$.  
\end{abstract}

\maketitle

%%%%%%%%%%%%%%%%%%%%%%%%%%%%%%%%%%%%%%%%%%%%%%%%%%%%%%%%%%
%%% SECTION: Introduction %%%%%%%%%%%%%%%%%%%%%%%%%%%%%%%%
%%%%%%%%%%%%%%%%%%%%%%%%%%%%%%%%%%%%%%%%%%%%%%%%%%%%%%%%%%

\section{Introduction}

The notion of \emph{continuous orbit equivalence}, pioneered by 
Matsumoto (\cite{matsu}), has proven to be an extremely important vehicle for understanding the relationship between dynamical systems and the $C^*$-algebras that they define. Indeed, this concept was a key ingredient which allowed Matsumoto and Matui to prove that the stabilized Cuntz-Krieger algebras become complete invariants for flow equivalence of irreducible shifts of finite type when considered not just as $C^*$-algebras, but as $C^*$-algebras containing a canonical commutative subalgebra, the \emph{diagonal}. The key result in \cite{MR3276420} thus gave an extremely elegant answer to the question that has been left open since R\o rdam 
as a key step in the proof of his seminal classification result \cite{MR1340839} showed that two such shift spaces can give the same Cuntz-Krieger algebra without being flow equivalent: The diagonal is precisely the extra structure which is needed for the $C^*$-algebra to remember its underlying shift space.

The success of the approach of Matsumoto and Matui begs the question of whether or not similar results hold true for more general $C^*$-algebras such as non-simple Cuntz-Krieger algebras and (simple or non-simple) graph $C^*$-algebras, objects which are currently (see \cite{preprint:ERRS}) giving way to classification in a way parallelling R\o rdam's results. In a sweeping generalization of Matsumoto's fundamental result, Brownlowe, Carlsen and Whittaker in \cite{arXiv:1410.2308} showed that continuous orbit equivalence exactly translates to diagonal-preserving isomorphism of the graph $C^*$-algebras 
for any graph with the so-called \emph{Condition (L)}, and proved by example that this condition is necessary. 

In the paper at hand, we will study \emph{continuous orbit equivalences preserving eventually periodic points} and prove that they exactly correspond to diagonal-preserving isomorphism of graph $C^*$-algebras. This small adjustment of the notions studied in \cite{MR3276420} and \cite{arXiv:1410.2308} thus allow a complete understanding also when Condition (L) fails.  In particular, we prove that the original notion of orbit equivalence corresponds to diagonal-preserving isomorphism when the graphs are finite and have no sinks, a case  prominently containing  the classical Cuntz-Krieger case. 
%In forthcoming work by the second named author and Carlsen, Ortega and Restorff, we will show how it may then be inferred by further generalizations of \cite{MR3276420} that in fact any diagonal-preserving stable isomorphism between Cuntz-Krieger algebras come from flow equivalence of the underlying shift spaces in complete generality.

Our method of proof involves  reducing the general case to the Condition (L) case and hence to the  main result of \cite{arXiv:1410.2308} by an elaboration the concept of ``plugging'' and ``unplugging'' graphs introduced in \cite{preprint:ERRS}.

After having circulated an early version of this paper, we were made aware that our main result had been simultaneously obtained by Carlsen and Winger (\cite{cw}) by completely different methods.

%%%%%%%%%%%%%%%%%%%%%%%%%%%%%%%%%%%%%%%%%%%%%%%%%%%%%%%%%%
%%% SECTION: Preliminaries %%%%%%%%%%%%%%%%%%%%%%%%%%%%%%%%
%%%%%%%%%%%%%%%%%%%%%%%%%%%%%%%%%%%%%%%%%%%%%%%%%%%%%%%%%%

\section{Preliminaries}

In this section, we provide the definitions of the objects considered in this paper.  We start with some background on directed graphs, graph $C^*$-algebras and their diagonal subalgebras.  The definitions of the boundary path space and the graph groupoid of a directed graph are also provided.

%%%%%%%%%%%%%%%%%%%%%%%%%%%%%%%%%%%%%%%%%%%%%%%%%%%%%%%%%%
%%% SUBSECTION: graph C*-algebras %%%%%%%%%%%%%%%%%%%%%%%%%%%%%%%%
%%%%%%%%%%%%%%%%%%%%%%%%%%%%%%%%%%%%%%%%%%%%%%%%%%%%%%%%%%

\subsection{Graph $C^*$-algebras and the diagonal subalgebra}

A \emph{directed graph} $E = (E^0 , E^1, r, s )$ consists of sets $E^0$ and $E^1$ and functions, $\ftn{r,s}{E^1}{E^0}$ called the range and source maps, respectively.  The elements of $E^0$ are called the \emph{vertices} of $E$ and the elements of $E^1$ are called the \emph{edges} of $E$.

\begin{assumption}
Throughout the paper, unless stated otherwise, when we say a graph we mean a directed graph.  Moreover, we will only consider graphs such that the set of vertices and the set of edges are countable sets.
\end{assumption}

Let $E$ be a graph.  A \emph{path of length $n$} in $E$ is a finite sequence $\mu = e_1 e_2 \dots e_n$ with $e_i \in E^1$ and $r( e_i ) = s( e_{i+1} )$ for all $i = 1, \dots, n-1$.  We will regard the vertices $E^0$ of $E$ as paths of length zero.  Denote the set of paths of length $n$ in $E$ by $E^n$.  Set $E^* = \bigcup_{ n \in \N_0 } E^n$ and set $E^{\geq k } = \bigcup_{ n \geq k } E^n$.  We extend the range and source maps to $E^*$ by $r(v) = s(v) = v$ for $v \in E^0$, and $s( e_1 \cdots e_n ) = s( e_1 )$ and $r( e_1 \cdots e_n ) = r( e_n )$.   

A \emph{loop} in $E$ is an edge $e$ in $E$ such that $s(e) = r(e)$.  A \emph{cycle} in $E$ is a path $\mu \in E^{\geq 1}$ such that $s( \mu ) = r( \mu )$.  A cycle $e_1 e_2 \cdots e_n$ in $E$ is said to have an \emph{exit} if there exists an $f \in E^1$ such that $s(e_k) = s(f)$ for some $k = 1, 2, \dots, n$ with $f \neq e_k$.  A \emph{vertex-simple cycle} in $E$  is a cycle $\mu = e_1 e_2 \cdots e_n$ such that $r( e_i ) \neq r( e_j )$ for all $i \neq j$.  A \emph{return path} in $E$ is a cycle $\mu = e_1 e_2 \cdots e_n$ such that $r( e_i ) \neq r( \mu )$ for all $i = 1, 2, \dots, n - 1$.  

An \emph{infinite path} in $E$ is an infinite sequence $( e_n )_{ n = 1 }^\infty$, denoted $e_1 e_2 \cdots$, such that $e_i \in E^1$ and $r(e_i) = s(e_{i+1})$ for all $i$.  The set of infinite paths in $E$ will be denoted by $E^\infty$.  If $e$ is a loop and $n \in \N$, then $e^n$ will denote the cycle of length $n$ with edges equal to $e$ and $e^\infty$ will denote the infinite path with edges equal to $e$.
If $\mu$ is a cycle in $E$, then $\mu^\infty$ denotes the infinite path $\mu\mu\mu\cdots$.

\begin{definition}
A graph $E$ is said to have \emph{Condition~(L)} if every cycle in $E$ has an exit.  
\end{definition}

Let $V, W$ be subsets of $E^0$, $S$ a subset of $E^* \cup E^\infty$, and $n \in \N_0$.  Define subsets of $E^* \cup E^\infty$, labelled $VS$, $SW$, and $VSW$, by $V S = \setof{ \mu \in S }{ s(\mu) \in V}$, $S W = \setof{ \mu \in S }{ r(\mu) \in W }$, and $V S W = V S \cap S W$.  Note that $SW$ and $VSW$ are subsets of $E^*$ since the range map is only defined on $E^*$.  We will write $v S$ if $V = \{ v \}$.  Similarly for $SW$ and $VSW$.  A vertex $v \in E^0$ is called \emph{regular} if $0 < | v E^1 | < \infty$.  Denote the set of regular vertices in $E$ by $E_\reg^0$.  A \emph{singular vertex} is a vertex $v$ in $E$ that is not regular.  We denote the set of singular vertices by $E_\sing^0$.  A  vertex $v$ in $E$ is called a \emph{sink} if $| v E^1 | = 0$ and is called an \emph{infinite emitter} if $| v E^1 | = \infty$.  Denote the set of sinks by $E^0_\sink$ and the set of infinite emitters by $E^0_\mathrm{inf}$.  Hence, $E^0_\mathrm{sing} = E^0_\sink \cup E^0_\mathrm{inf}$.

We call an infinite path $e_1e_2\cdots$ in $E$ a \emph{tail} if $s(e_i)E^1=\{e_i\}=E^1s(e_{i+1})$, and \emph{non-wandering} if $s(e_i)E^1=\{e_i\}$.
If $\mu\in E^*$ is a cycle with no exits, then $\mu^\infty$ is non-wandering.

\begin{definition}\label{d:graphalgebras}
Let $E$ be a graph.  A \emph{Cuntz-Krieger $E$-family} in a $C^*$-algebra $A$ consists of a set of mutually orthogonal projections $\setof{ P_v }{ v \in E^0 } \subseteq A$ and a set of partial isometries $\setof{ S_e }{ e \in E^1 } \subseteq A$ satisfying 
\begin{enumerate}[label=\textbf{(CK\arabic*)}]
\item \label{itm:ck1} $S_e^* S_f = 0$ for all $e, f \in E^1$ with $e \neq f$;

\item \label{itm:ck2} $S_e^* S_e = P_{r(e)}$ for all $e \in E^1$;

\item \label{itm:ck3} $S_e S_e^* \leq P_{s(e)}$ for all $e \in E^1$; and 

\item \label{itm:ck4} $P_v = \sum_{ e \in v E^1 } S_e S_e^*$ for all $v \in E^0_\reg$.
\end{enumerate}
The \emph{graph $C^*$-algebra $C^*(E)$} is the universal $C^*$-algebra generated by a Cuntz-Krieger $E$-family.  

If $\mu= e_1 e_2 \cdots e_n \in E^{\geq 2}$, we set $s_\mu = s_{e_1} s_{e_2} \cdots s_{e_n}$ and for $v \in E^0$, we set $s_v = p_v$.  Then the $C^*$-subalgebra $\overline{\mathrm{span}} \setof{ s_\mu s_\mu^* }{ \mu \in E^* }$ of $C^*(E)$ is called the \emph{diagonal subalgebra} of $C^*(E)$ and is denoted by $\mathcal{D}(E)$.  
\end{definition}

\begin{definition}
Let $E$ and $F$ be graphs.  A \stariso $\ftn{\Phi}{C^*(E)}{C^*(F)}$ is a \emph{diagonal preserving \stariso} if $\Phi ( \mathcal{D} (E) ) = \mathcal{D} (F)$.
\end{definition}

%%%%%%%%%%%%%%%%%%%%%%%%%%%%%%%%%%%%%%%%%%%%%%%%%%%%%%%%%%
%%% SUBSECTION: boundary path %%%%%%%%%%%%%%%%%%%%%%%%%%%%%%%%
%%%%%%%%%%%%%%%%%%%%%%%%%%%%%%%%%%%%%%%%%%%%%%%%%%%%%%%%%%

\subsection{Boundary path space and the graph groupoid of a graph}  The definitions given in this section follows that of \cite[Section~2.2 and 2.3]{arXiv:1410.2308}.
 
\begin{definition}\label{d:boundary}
Let $E$ be a graph.  The \emph{boundary path space} of $E$ is the space
\[
\partial E = E^\infty \sqcup \setof{ \mu \in E^* }{ r(\mu ) \in E_{\mathrm{sing}}^0 }.
\]
The topology on $\partial E$ is given as follows:  For $\mu \in E^*$, the \emph{cylinder set} of $\mu$ is the set 
\[
\mathcal{Z}_E ( \mu ) = \setof{ \mu x \in \partial E }{ x \in r(\mu)\partial E },
\]
where $\mu x$ is the concatenation of paths.  For $\mu \in E^*$ and a finite subset $F$ of $r( \mu ) E^1$, set 
\[
\mathcal{Z}_E ( \mu \setminus F ) = \mathcal{Z}_E (\mu ) \setminus \left( \bigcup_{ e \in F } \mathcal{Z}_E ( \mu e ) \right).
\]
(When there is no cause for confusion, we will at times omit the subscripts.)  The topology of $\partial E$ is the topology generated by 
\[
\mathcal{B}_E = \setof{ \mathcal{Z}_E ( \mu \setminus F ) }{ \text{$\mu \in E^*$, $F$ a finite subset of $r(\mu)E^1$} }. 
\]
\end{definition} 

The boundary path space $\partial E$ is a locally compact Hausdorff space with basis $\mathcal{B}_E$ and every $U \in \mathcal{B}_E$ is compact and open (see \cite[Theorem~2.1 and Theorem~2.2]{MR3119197}).  

The key relationship between $\partial E$ and $\mathcal{D}(E)$ is the following theorem.

\begin{theorem}[{\cite[Theorem~3.7]{MR3119197}}]\label{thm:prim-diagonal}
There exists a unique homeomorphism $h_E$ from $\partial E$ to the spectrum of $\mathcal{D} (E)$ given by 
\[
h_E (x) ( s_\mu s_\mu^* ) = 
\begin{cases}
1 &\text{if $x \in \mathcal{Z}(\mu)$} \\
0 &\text{otherwise}.
\end{cases}
\]
\end{theorem}

\begin{lemma}\label{l:sinks-closed}
Let $E$ be a graph and let $S$ be a subset of $E^0_\sink$.  Then $S$ is a clopen (i.e., closed and open) subset of $\partial E$.
\end{lemma}

\begin{proof}
First note that for all $v \in E^0_\sink$ and for all $\mu \in E^*$ and $F$ finite subset of $r(\mu ) E^1$, $v \in \mathcal{Z}_E ( \mu \setminus F )$ if and only if $v = \mu$ and $F = \emptyset$.  In particular, for all $v \in E^0_\sink$, $\mathcal{Z}_E ( v ) = \{ v \}$.  Therefore, $S = \bigcup_{ v \in S } \mathcal{Z}_E (v)$, and hence open.  Using again that the only cylinder set that contains a sink $v$ is $\{v\}$, we conclude that if $x \in \partial E \setminus S$, 
%and $\mu \in E^*$ and a finite subset $F$ of $r(\mu) E^1$ such that $x \in \mathcal{Z}_E ( \mu \setminus F )$, then $\mathcal{Z}_E ( \mu \setminus F ) \subseteq \partial E \setminus S$.  
then any cylinder set $\mathcal Z_E(\mu)$ containing $x$ will satisfy $\mathcal Z_E(\mu)\subseteq\partial E\setminus S$.
So, $\partial E \setminus S$ is open which implies that $S$ is closed.
\end{proof}

If $x \in \partial E$, then set
\[
|x| = 
\begin{cases}
\infty &\text{if $x \in E^\infty$} \\
n &\text{if $x \in E^n$ for some $n \in \N_0$}.
\end{cases}
\]
For $n \in \N_0$, set $\partial E^{\geq n} = \setof{ x \in \partial E }{ |x| \geq n }$.  Note that $\partial E^{\geq n}$ is an open subset of $\partial E$ since $\partial E^{\geq n}$ is equal to $\bigcup_{ \mu \in E^n } \mathcal{Z} ( \mu )$.

\begin{definition}\label{d:shiftmap}
Let $E$ be a graph.  Define the \emph{shift map} $\ftn{\sigma_E}{ \partial E^{\geq 1} }{ \partial E }$ on $E$ by 
\[
\sigma_E ( x ) = 
\begin{cases}
e_2 e_3 \cdots &\text{if $x = e_1 e_2 \cdots \in \partial E^{\geq 2}$} \\
r(x) &\text{if $x \in \partial E \cap E^1$}.
\end{cases}
\]
For $n \geq 1$, $\ftn{\sigma_E^n}{ \partial E^{\geq n } }{ \partial E}$ will be the $n$-fold composition of $\sigma_E$ with itself and $\ftn{\sigma_E^0}{\partial E}{\partial E}$ will be the identity map.   
\end{definition}

One can check that for all $n \in \N_0$,  $\sigma_E^n$ is continuous and moreover, $\sigma_E^n$ is a local homeomorphism.

We now define the graph groupoid of a graph $E$.  

\begin{definition}\label{d:graphgroupoid}
Let $E$ be a graph.  The \emph{graph groupoid} $\mathcal{G}_E$ is defined as follows:  As a set, 
\[
\mathcal{G}_E = \setof{ ( x, m - n , y ) }{ \text{$x, y \in \partial E$ with $| x | \geq m$, $|y| \geq n$, and $\sigma^m_E (x) = \sigma^n_E(y)$} }. 
\]
The product is defined by $( x, k , y ) ( w ,l , z ) = ( x, k + l , z )$ if $y = w$ and undefined otherwise, and the inverse of $(x,k, y)$ is $( y, -k, x )$.  The set of units $\mathcal{G}_E^{(0)}$ of $\mathcal{G}_E$ is $\setof{ (x, 0 , x ) }{ x \in \partial E }$.    

Let $m, n \in \N_0$, $U$ be an open subset of $\partial E^{ \geq m }$ such that $\sigma^m_E \vert_U$ is injective, and $V$ be an open subset of $\partial E^{\geq n}$ such that $\sigma^n_E \vert_V$ is injective.  Suppose $\sigma_E^m (U) = \sigma_E^n (V)$.  Set
\[
\mathcal{Z}_E (U, m, n, V ) = \setof{ (x, m-n, y) \in \mathcal{G}_E }{ x \in U, y\in V, \sigma^{m}_E (x) = \sigma_E^n(y) }. 
\]
(When there is no cause for confusion, we will at times omit the subscripts.)  Then $\mathcal{G}_E$ is a locally compact, Hausdorff, \'etale topological groupoid with the topology generated by the basis consisting of sets $\mathcal{Z}_E (U, m, n, V )$.
\end{definition}

One checks that the map $\mu \in \partial E \mapsto (\mu , 0 , \mu ) \in \mathcal{G}_E^{(0)}$ is a homeomorphism from $\partial E$ to $\mathcal{G}_E^{(0)}$.  We will freely identify $\mathcal{G}_E^{(0)}$ with $\partial E$ using this map throughout the paper without further mention.  Thus, we have range and source maps $\ftn{r,s}{ \mathcal{G}_E }{ \partial E}$ defined by $r( (x, k, y ) ) = x$ and $s( (x, k, y ) ) = y$.

By \cite[Proposition~3.3.5 and 6.1.8]{MR1799683}, the reduced and universal $C^*$-algebra of $\mathcal{G}_E$ are equal since $\mathcal{G}_E$ is topologically amenable (see \cite[Proposition~6.2]{MR2301938}).  We denote this $C^*$-algebra by $C^*( \mathcal{G}_E)$.

\begin{theorem}[{\cite[Proposition~2.2]{arXiv:1410.2308}}]
If $E$ is a graph, then there exists a unique  \stariso $\ftn{ \Phi_E }{ C^*(E) }{ C^*( \mathcal{G}_E )}$ such that $\Phi_E ( p_v ) = 1_{\mathcal{Z}(v,v)}$ for all $v \in E^0$ and $\Phi_E ( s_e ) = 1_{ \mathcal{Z} ( e, r(e) ) }$ for all $e \in E^1$, and such that $\Phi_E ( \mathcal{D} (E) ) = C_0 ( \mathcal{G}_E^{(0)} )$. 
\end{theorem}

%%%%%%%%%%%%%%%%%%%%%%%%%%%%%%%%%%%%%%%%%%%%%%%%%%%%%%%%%%
%%% SECTION: Orbit equivalence and pseudogroups %%%%%%%%%%%%%%%%%%%%%%%%%%%%%%%%
%%%%%%%%%%%%%%%%%%%%%%%%%%%%%%%%%%%%%%%%%%%%%%%%%%%%%%%%%%

\section{Orbit equivalence preserving periodic points and pseudogroups}

We now define orbit equivalence between graphs $E$ and $F$ preserving periodic points.  When $E$ and $F$ are graphs with finitely many vertices and no sinks, or when $E$ and $F$ are graphs satisfying Condition~(L), then this notion of orbit equivalence coincides with the notion of orbit equivalence defined in \cite[Definiton~3.1]{arXiv:1410.2308}.   

%%%%%%%%%%%%%%%%%%%%%%%%%%%%%%%%%%%%%%%%%%%%%%%%%%%%%%%%%%
%%% SUBSECTION: orbit equivalence %%%%%%%%%%%%%%%%%%%%%%%%%%%%%%%%
%%%%%%%%%%%%%%%%%%%%%%%%%%%%%%%%%%%%%%%%%%%%%%%%%%%%%%%%%%

\subsection{Orbit equivalence preserving periodic points}  Let $E$ be a graph, and let $\partial E_\iso$ denote the set of isolated points in $\partial E$.  
%One checks that for $x \in \partial E$, if $r(x) \in E^0_\sink$ or $x = \mu \nu\nu \nu \cdots$ where $\nu$ is a vertex-simple cycle with no exits, then $x \in \partial E_\iso$.  

\begin{definition}
Let $E$ be a graph.  Then $x \in \partial E$ is \emph{eventually periodic} if $x =\mu \nu^\infty$ where $\mu \in E^*$, $\nu$ is a cycle in $E$, and $r(\mu ) = s( \nu )$.  
\end{definition}

Note that an eventually periodic point $x\in\partial E$ is an isolated point if and only if $x=\mu\nu^\infty$ where $\mu\in E^*$ and $\nu\in E^*$ is a cycle with no exits satisfying $r(\mu)=s(\nu)$.

We call $x\in\partial E$ \emph{eventually a sink} if $x\in E^*$ with $r(x)\in E^0_\sink$, and call $x$ \emph{eventually non-wandering} if $x=\mu y$ where $\mu\in E^*$ and $y\in r(\mu)E^\infty$ is non-wandering.
The isolated points in $\partial E$ are exactly the points that are eventually sinks or eventually non-wandering.
Clearly, the eventually periodic isolated points are exactly the eventually periodic eventually non-wandering points.
We will refer to the eventually non-wandering points that are are not eventually periodic as \emph{eventually non-periodic non-wandering}.

\begin{definition}\label{d:orbitequivalence}
Let $E$ and $F$ be graphs and let $\ftn{ \kappa }{ \partial E }{ \partial F }$ be a homeomorphism.  We say that $\kappa$ is an \emph{orbit equivalence} if there exist continuous functions $\ftn{ l,m }{\partial E^{ \geq 1 } }{ \N_0 }$ and $\ftn{ l',m' }{\partial F^{\geq 1 }}{\N_0}$ such that 
\begin{enumerate}
\item \label{def:oe1} $\sigma_F^{m(x)} ( \kappa( \sigma_E( x) ) ) = \sigma_F^{l(x)} ( \kappa(x) )$ for all $x \in \partial E^{\geq 1}$; and

\item \label{def:oe2} $\sigma_E^{m'(y)} ( \kappa^{-1}( \sigma_F( y) ) ) = \sigma_E^{l'(y)} ( \kappa^{-1}(y) )$ for all $y \in \partial F^{\geq 1}$.
\end{enumerate}
If such a $\kappa$ exists, we say that $E$ and $F$ are \emph{orbit equivalent} or there exists an \emph{orbit equivalence between $E$ and $F$}.

If, in addition, $\kappa$ satisfies
\begin{enumerate}[(3)]
\item \label{def:oe3} for all $x \in \partial E_{\iso}$, $x$ is eventually periodic if and only if $\kappa(x)$ is eventually periodic,
\end{enumerate}
then we say that $\kappa$ is an \emph{orbit equivalence preserving periodic points}.  If such a $\kappa$ exists, then we say that there exists an \emph{orbit equivalence between $E$ and $F$ preserving periodic points}.
\end{definition}

All eventually non-wandering points in $\partial E$ are eventually periodic if the graph $E$ has finitely many vertices.
Hence if $E$ and $F$ are graphs with finitely many vertices and no sinks, an orbit equivalence between $E$ and $F$ will automatically preserve periodic points.
Likewise, if $E$ and $F$ are graphs satisfying Condition (L), any orbit equivalence between $E$ and $F$ will preserve periodic points, as $E$ and $F$ contain no eventually periodic isolated points.

In general, if $E$ and $F$ are orbit equivalent graphs, there may not exist an orbit equivalence between $E$ and $F$ preserving periodic points, as all three types of isolated points may be interchanged by an orbit equivalence.
%Note that if there exists an orbit equivalence between $E$ and $F$ preserving periodic points then $E$ and $F$ are orbit equivalent.  The converse is not true, as all three types of isolated points may be interchanged.
Consider the graphs
%\medskip
\[
 \xymatrix{ E \colon &  \bullet } \quad \quad  \text{and} \quad \quad \xymatrix{ F \colon & \bullet \ar@(ul,ur)[]^-{} }.
\]
By \cite[Example~5.2]{arXiv:1410.2308}, $E$ and $F$ are orbit equivalent but there is no orbit equivalence between $E$ and $F$ preserving periodic points,
as the isolated point in $\partial E$ is a sink while the isolated point in $\partial F$ is periodic.
In Example~\ref{ex:interchanging_isolated} we provide examples of orbit equivalences that interchange eventually periodic points with eventually non-periodic non-wandering points, and eventually sinks with eventually non-periodic non-wondering points.

\begin{example} \label{ex:interchanging_isolated}
Consider the graphs $E$, $F$ and $G$:
\[
 \xymatrix{ 
 & \vdots\ar[d]^-{e_2} \\
 E \colon & \bullet\ar[d]^-{e_1} \\
 &  \bullet \ar@(dr,dl)[]^-{f} 
 }  \quad  \quad \quad \xymatrix{ 
 & \bullet\ar[d]^-{e_1} \\
 F \colon & \bullet\ar[d]^-{e_2} \\
 & \vdots
}  \quad   \quad \quad \xymatrix{ 
& \vdots\ar[d]^-{e_2} \\
G \colon  & \bullet \ar[d]^-{e_1} \\
& v %\ar[d]^-{f_1} \\
%& \bullet \ar[d]^-{f_2} \\
%& \vdots
 }
\]
Then all points in $\partial E$ are eventually periodic isolated points, all points in $\partial F$ are eventually non-periodic non-wandering, and all points in $\partial G$ are eventually sinks.
We now show that $E$, $F$, and $G$ are orbit equivalent.

Define $\kappa_1\colon\partial E\to\partial F$ and $\kappa_2\colon\partial F\to\partial G$ by $\kappa_1(e_i\cdots e_1f^\infty) = e_{i+1}e_{i+2}\cdots $ for $i\geq 1$ and $\kappa_1(f^\infty)=e_1e_2\cdots$, and $\kappa_2(e_ie_{i+1}\cdots)=e_{i-1}\cdots e_1$ for $i\geq 2$ and $\kappa_2(e_1e_2\cdots)=v$.
Clearly, $\kappa_1$ and $\kappa_2$ are bijective, and they are continuous and open since $\partial E$, $\partial F$, and $\partial G$ carry the discrete topologies.
One readily checks that $m_1,l_1\colon \partial E^{\geq 1}\to\N_0$ defined by $m_1(f^\infty)=0$, $m_1(e_i\cdots e_1f^\infty)=1$, and $l_1(x)=0$ for all $x\in\partial E^{\geq 1}$ satisfies
\[ \sigma_F^{m_1(x)}(\kappa_1(\sigma_E(x)))=\sigma_F^{l_1(x)}(\kappa_1(x)) \]
for all $x\in\partial E^{\geq 1}$, 
and that $m'_1,l'_1\colon \partial F^{\geq 1}\to\N_0$ defined by $m'_1(x)=1$ and $l'_1(x)=0$ satisfies
\[ \sigma_E^{m'_1(x)}(\kappa^{-1}_1(\sigma_F(x)))=\sigma_E^{l'_1(x)}(\kappa^{-1}_1(x)) \]
for all $x\in\partial F^{\geq 1}$.
So $\kappa_1$ is an orbit equivalence, since the maps $m_1,l_1,m'_1,l'_1$ are automatically continuous.
Similarly, $m_2,l_2\colon \partial F^{\geq 1}\to\N_0$ defined by $m_2(x)=1$ and $l_2(x)=0$, and $m'_2,l'_2\colon \partial G^{\geq 1}\to\N_0$ defined by $m'_2(x)=1$ and $l'_2(x)=0$, lets us conclude that $\kappa_2$ is an orbit equivalence.
\end{example}

%%%%%%%%%%%%%%%%%%%%%%%%%%%%%%%%%%%%%%%%%%%%%%%%%%%%%%%%%%
%%% SUBSECTION: Groupoid of germs %%%%%%%%%%%%%%%%%%%%%%%%%%%%%%%%
%%%%%%%%%%%%%%%%%%%%%%%%%%%%%%%%%%%%%%%%%%%%%%%%%%%%%%%%%%

\subsection{The pseudogroup $\mathcal{P}_E$ and the groupoid of germs of $\mathcal{P}_E$}
 
We now recall the groupoid of germs defined in \cite[Section~3]{MR2460017}.  

Let $X$ be a topological space.  A homeomorphism $\ftn{h}{U}{V}$ where $U$ and $V$ are open subsets of $X$ is called a \emph{partial homeomorphism}.  Under composition and inverse, the collection of partial homeomorphisms on $X$ is an inverse semigroup.  A \emph{pseudogroup} on $X$ is a family of partial homeomorphisms of $X$ stable under composition and inverse.  

Let $\mathcal{P}$ be a pseudogroup on $X$.  A partial homeomorphism $\ftn{h}{U}{V}$ is said to \emph{locally belong to $\mathcal{P}$} if for all $x \in U$, there exists an open neighborhood $W$ of $x$ and there exists $g \in \mathcal{P}$ such that $h \vert_W = g \vert_W$.  The pseudogroup $\mathcal{P}$ is \emph{ample} if each partial homeomorphism $\ftn{h}{U}{V}$ that locally belongs to $\mathcal{P}$ must also be element in $\mathcal{P}$.  

\begin{definition}
Let $\mathcal{P}$ be a pseudogroup on the topological space $X$.  The \emph{groupoid of germs} of $\mathcal{P}$ is 
\[
\mathcal{G}_{\mathcal{P}} = \setof{ [ x, h, y ] }{  h \in \mathcal{P}, y \in \mathrm{dom} ( h ), x = h (y) }
\]
where $[ x , h, y ] = [x , g , y ]$ if and only if there exists a neighborhood $V$ of $y$ in $X$ such that $h \vert_V = g \vert_V$.

The range and source maps are given by 
\[
r ([ x, h , y ] )= x \quad \text{and} \quad s([ x, h , y ]) = y.
\]
The partially defined product is $[x , h , y ][ y, g, z ] = [x , h \circ g , z ]$, undefined otherwise and the inverse $[x , h , y ]^{-1} = [ y, h^{-1} , x ]$.  The groupoid $\mathcal{G}_\mathcal{P}$ is given the topology given by basic open sets 
\[
\mathcal{U} ( U, h , V ) = \setof{ [ x , g, y ] \in \mathcal{G}_{\mathcal{P}} }{ x \in U, y \in V }
\]
where $U$ and $V$ are open subsets of $X$ and $h \in \mathcal{P}$.
\end{definition}

We recall how to construct a pseudogroup from an \'{e}tale groupoid $\mathcal{G}$.  A subset $A$ of a groupoid $\mathcal{G}$ is called a \emph{bisection} if $r\vert_A$ and $s\vert_A$ are injective functions.  Then the set of all open bisections $\mathcal{S}( \mathcal{G} )$ forms an inverse semigroup with composition law 
\[
AB = \setof{ \gamma \gamma' }{  ( \gamma , \gamma' )  \in ( A \times B ) \cap \mathcal{G}^{(2)} }
\]
and 
\[
A^{-1} = \setof{ \gamma^{-1} }{ \gamma \in A }.
\]

Let $A$ be an open bisection.  Then define $\ftn{ \alpha_A }{ s( A ) }{ r(A) }$ by $\alpha_A ( s(\gamma ) ) = r(\gamma)$ for all $\gamma \in A$.  One checks that $\alpha_A$ is a homeomorphism.   Then  the pseudogroup on $\mathcal{G}^{(0)}$ is  
\[
\mathcal{P}( \mathcal{G} ) = \setof{ \alpha_A }{ \text{$A$ is an open bisection} }.
\]

\begin{assumption}
Isomorphisms between topological groupoids are isomorphisms between groupoids that are also homeomorphisms.  
\end{assumption}

The following proposition follows from \cite[Proposition~3.6]{MR2460017} and the proofs of \cite[Proposition~3.2 and Corollary~3.3]{MR2460017}.
\begin{proposition}
Let $\mathcal{G}$ be an \'{e}tale groupoid.  Define $\ftn{\phi_\mathcal{G}}{\mathcal{G}}{\mathcal{G}_{ \mathcal{P} (\mathcal{G} ) }}$ by
\[
\phi_\mathcal{G} ( \gamma ) = [ r( \gamma ), \alpha_A , s( \gamma ) ]
\]
where $A$ is an open bisection containing $\gamma$.  Then $\phi_\mathcal{G}$ is a well-defined surjective homomorphism of groupoids.  Moreover, if $\mathcal{G}$ is Hausdorff and topologically principal, then $\phi_\mathcal{G}$ is an isomorphism.
\end{proposition}

As an immediate consequence, we get the following corollary.  

\begin{corollary}\label{cor:pseudogroup-groupoid-iso}
Let $\mathcal{G}$ and $\mathcal{H}$ be \'{e}tale groupoids.  Suppose there exists a homeomorphism $\ftn{\kappa}{\mathcal{G}^{(0)}}{\mathcal{H}^{(0)}}$ such that 
\[
\kappa \circ \mathcal{P}(\mathcal{G}) \circ \kappa^{-1} := \setof{ \kappa \circ g \circ \kappa^{-1} }{ g \in \mathcal{P}( \mathcal{G} ) } = \mathcal{P} (\mathcal{H}).
\]
Then there exists an isomorphism $\ftn{\psi_\kappa }{\mathcal{G}_{\mathcal{P}(\mathcal{G})} }{\mathcal{G}_{\mathcal{P}(\mathcal{H})}}$ defined by 
\[
\psi_\kappa ( [ x, g, y ] ) = [ \kappa(x) , \kappa \circ g \circ \kappa^{-1} , \kappa(y) ].
\]

Consequently, if $\mathcal{G}$ and $\mathcal{H}$ are Hausdorff and topological principal, then $\psi_\kappa$ induces an isomorphism $\ftn{ \widetilde{\psi}_\kappa}{ \mathcal{G} }{ \mathcal{H} }$ such that $\widetilde{\psi}_\kappa \vert_{ \mathcal{G}^{(0)} } = \kappa$.
\end{corollary}

Of particular interest to us is the pseudogroup of the \'{e}tale groupoid $\mathcal{G}_E$ for a graph $E$.  So, for a graph $E$, we denote $\mathcal{P} ( \mathcal{G}_E )$ by $\mathcal{P}_E$ and we call $\mathcal{P}_E$ the \emph{pseudogroup of $E$}.

In \cite[Proposition~3.4]{arXiv:1410.2308}, the authors prove that $E$ and $F$ are orbit equivalent if and only if the pseudogroups of $E$ and $F$ are isomorphic, i.e., there exists a homeomorphism $\ftn{\kappa}{ \partial E }{ \partial F }$ such that 
\[
\kappa \circ \mathcal{P}_E \circ \kappa_E^{-1} = \setof{ \kappa \circ g \circ \kappa^{-1} }{ g \in \mathcal{P}_E } = \mathcal{P}_F.
\]  
They actually prove a stronger statement in the sense that the orbit equivalence between $E$ and $F$ induces the isomorphism between the pseudogroups of $E$ and $F$ and vice versa.  We record this in the following proposition.

\begin{proposition}[{\cite[Proposition~3.4]{arXiv:1410.2308}}]\label{prop:arXiv:1410.2308:Prop3.4}
Let $E$ and $F$ be graphs and $\kappa$ from $\partial E$ to $\partial F$ be a homeomorphism.  Then $\kappa$ is an orbit equivalence (preserving periodic points) if and only if $\kappa \circ \mathcal{P}_E \circ \kappa^{-1} = \mathcal{P}_F$ (and for all $x \in \partial E_{\iso}$, $x$ is eventually periodic if and only if $\kappa(x)$ is eventually periodic).  
\end{proposition}

\begin{proposition}\label{prop:orbitequivalence-equivalencerelation}
For all graphs $E_1$, $E_2$, and $E_3$, if $\kappa_1$ from $\partial E_1$ to $\partial E_2$ and $\kappa_2$ from $\partial E_2$ to $\partial E_3$ are orbit equivalences (preserving periodic points), then $\kappa_2 \circ \kappa_1$ from $\partial E_1$ to $\partial E_3$ is an orbit equivalence (preserving periodic points).
\end{proposition}

\begin{proof}
Let $E_1$, $E_2$, and $E_3$ be graphs.  Suppose $\ftn{\kappa_1}{\partial E_1}{\partial E_2}$ and $\ftn{\kappa_2}{\partial E_2}{\partial E_3}$ are orbit equivalences.  By Proposition~\ref{prop:arXiv:1410.2308:Prop3.4}, we have $\kappa_1 \circ \mathcal{P}_{E_1} \circ \kappa_1^{-1} = \mathcal{P}_{E_2}$ and $\kappa_2 \circ \mathcal{P}_{E_2} \circ \kappa_2^{-1} = \mathcal{P}_{E_3}$.  It follows that $( \kappa_2 \circ \kappa_1 ) \circ \mathcal{P}_{E_1} \circ  ( \kappa_2 \circ \kappa_1 )^{-1}   = \mathcal{P}_{E_3}$.  Thus, by Proposition~\ref{prop:arXiv:1410.2308:Prop3.4}, $\ftn{ \kappa_2 \circ \kappa_1 }{ \partial E_1 }{ \partial E_3 }$ is an orbit equivalence.  

Suppose $\kappa_1$ and $\kappa_2$ are orbit equivalences preserving periodic points.  Then, as above $( \kappa_2 \circ \kappa_1 ) \circ \mathcal{P}_{E_1} \circ  ( \kappa_2 \circ \kappa_1 )^{-1}   = \mathcal{P}_{E_3}$.  Moreover, since $\kappa_1$ and $\kappa_2$ are orbit equivalences preserving periodic points, for all isolated points $x$ in $\partial E_1$, $x$ is eventually periodic if and only if $\kappa_1(x)$ is eventually periodic if and only if $(\kappa_2 \circ \kappa_1)(x)$ is eventually periodic.

\end{proof}

We are now able to prove a stronger version of \cite[Theorem~5.1]{arXiv:1410.2308}.  This version will be important for us in the proof of Theorem~\ref{thm:main-result}.

\begin{theorem}[{cf.~\cite[Theorem~5.1]{arXiv:1410.2308}}]\label{thm:arXiv:1410.2308}
Let $E$ and $F$ be graphs satisfying Condition~(L).  Suppose $\ftn{\kappa}{\partial E}{\partial F}$ is an orbit equivalence from $E$ to $F$.  Then there exists an isomorphism $\ftn{\phi}{\mathcal{G}_E}{\mathcal{G}_F}$ such that $\phi \vert_{\partial E} = \kappa$.  
\end{theorem}

\begin{proof}
By Proposition~\ref{prop:arXiv:1410.2308:Prop3.4}, $\kappa \circ \mathcal{P}_E \circ \kappa^{-1} = \mathcal{P}_F$.  Since $E$ and $F$ satisfy Condition~(L), by \cite[Proposition~2.3]{arXiv:1410.2308}, $\mathcal{G}_E$ and $\mathcal{G}_F$ are topologically principal.  By Corollary~\ref{cor:pseudogroup-groupoid-iso}, there exists an isomorphism  $\ftn{\phi}{\mathcal{G}_E}{\mathcal{G}_F}$ such that $\phi \vert_{ \partial E } = \kappa$.
\end{proof}

%%%%%%%%%%%%%%%%%%%%%%%%%%%%%%%%%%%%%%%%%%%%%%%%%%%%%%%%%%
%%% SUBSECTION: unplug %%%%%%%%%%%%%%%%%%%%%%%%%%%%%%%%
%%%%%%%%%%%%%%%%%%%%%%%%%%%%%%%%%%%%%%%%%%%%%%%%%%%%%%%%%%

\subsection{The unplugged graph and orbit equivalence}\label{converse}

For a graph $E$, let $E^0_{\cycle}$ be the set of vertices of $E$ that is
on a vertex-simple cycle with no exits.  Suppose $E$ satisfies the property that if $\nu$ is a vertex-simple cycle with no exits, then $\nu$ is a loop.  This entails that every vertex $v \in E^0_{\cycle}$ supports a unique loop $e_v$.  Note that if $v \in E^0_{\cycle}$ and $e \in E^1$
such that $s( e ) = v$, then $e = e_v$.  Denote the set of all loops based at a vertex in $E^0_\cycle$ by $E^1_\cycle$.  Note that $s ( E^1_\cycle ) = r ( E^1_\cycle ) = E^0_\cycle$.  Also note that if $e, f \in E^1_\cycle$ with $s(e) = s( f)$ (equivalently, $r(e) = r( f)$), then $e = f$.  

Let $E$ be a graph such that all vertex-simple cycles with no exits are loops.  Let the \emph{unplugged graph} $E_\curlyvee$ of $E$ be the graph defined by 
\[
E^0_\curlyvee = E^0 \quad \text{and} \quad E^1_\curlyvee = E^1 \setminus E^1_\cycle
\]
with the range and source maps of $E_\curlyvee$  the restrictions of the range and source maps of $E$ respectively.

\begin{proposition}\label{prop:oe-unplug}
Let $E$ be a graph such that each vertex-simple cycle with no exits is a loop.  Define $\ftn{\kappa_E}{\partial E_{\curlyvee}}{\partial E}$ by 
\[
\kappa_E (x) = 
\begin{cases}
x e_{r(x)}^\infty, &\text{if $x \in E^*_\curlyvee$ with $r(x) \in E^0_\cycle$} \\
x, &\text{otherwise.}
\end{cases}
\]
Then $\kappa_E$ is an orbit equivalence such that for each isolated point $x \in \partial E_\curlyvee$, $r_{E_\curlyvee}(x) \in E^0_{\cycle}$ if and only if $\kappa_E(x)$ is an isolated point in $\partial E$ that is eventually periodic.
\end{proposition}

\begin{proof}
A computation shows that $\kappa_E$ is a bijection with $\kappa_E^{-1}(x)=\mu$ when $x=\mu e_{r(\mu)}^\infty$ for some $\mu=e_1\cdots e_n$ with $r(\mu)\in E_\cycle^0$ and $e_n\neq e_{r(\mu)}$, and $\kappa_E^{-1}(x)=x$ for all other $x$.  
Let $\mu = e_1 e_2 \cdots e_n \in E^*$.  
Suppose $\mu\in E^*_\curlyvee$.  Then
%Suppose for each $i$, $e_i \notin E^1_\cycle$.  Then $\mu \in E^{*}_{\curlyvee}$.  Hence
\[
\kappa_E^{-1} ( \mathcal{Z}_E(\mu) ) = \setof{ \mu x \in \partial E_{\curlyvee} }{ x \in r(\mu) \partial E_{\curlyvee} } = \mathcal{Z}_{E_\curlyvee} (\mu)
\]
which is open.
Suppose $\mu\notin E^*_\curlyvee$.
%Suppose there exists $1 \leq i \leq n$ such that $e_i \in E^1_\cycle$.  
Let $i_0 \in \N$ with $1 \leq i_0 \leq n$ such that $e_{i_0} \in E^1_\cycle$ and $e_j \notin E^1_\cycle$ for all $j < i_0$.  
Then $e_{i_0} = e_k$ for all $i_0 \leq k \leq n$.  
Therefore, $\mathcal{Z}_E ( \mu ) = \mathcal{Z}_E ( e_1 \cdots e_{i_0-1} )=\{e_1\cdots e_{i_0-1}e_{i_0}^\infty\}$, where $ e_1 \cdots e_{i_0-1} = s(\mu)$ if $i_0 = 1$.  Hence, 
\[
\kappa_E^{-1} ( \mathcal{Z}_E(\mu) ) = %\kappa_E^{-1}( \mathcal{Z}_E ( e_1 \cdots e_{i_0-1} ) ) 
\kappa_E^{-1}(\{e_1\cdots e_{i_0-1}e_{i_0}^\infty\}) = \{e_1\cdots e_{i_0-1}\}
= \mathcal{Z}_{E_\curlyvee} ( e_1 \cdots e_{i_0-1})
\]
which is open.
Let $F$ be a finite nonempty subset of $r(\mu)E^1$.
If $r(\mu)\in E^0_\cycle$ then $F=\{e_{r(\mu)}\}$ so $\mathcal Z_E(\mu\setminus F)=\emptyset$, hence $\kappa_E^{-1} \left( \mathcal{Z}_E ( \mu \setminus F) \right)$ is trivially open.
If $r(\mu)\notin E^0_\cycle$ then $\mu\in E^*_\curlyvee$ and $F\subseteq r(\mu)E^1_\curlyvee$.  Hence
\begin{align*}
\kappa_E^{-1} \left( \mathcal{Z}_E ( \mu \setminus F) \right) &= \kappa_E^{-1}( \mathcal{Z}_E ( \mu ) ) \setminus \left( \bigcup_{ e \in F } \kappa_E^{-1} ( \mathcal{Z}_E(\mu e))\right) \\
&= \mathcal{Z}_{E_\curlyvee} ( \mu )  \setminus \left( \bigcup_{ e \in F }  \mathcal{Z}_{E_\curlyvee}(\mu e)\right)
= \mathcal{Z}_{E_\curlyvee} ( \mu \setminus F)
\end{align*}
which is open.
%Let $F$ be a finite subset of $r( \mu ) E^1$ and $\mathcal{Z}_E ( \mu \setminus F) \neq \emptyset$.  Then
%\[
%\kappa_E^{-1} \left( \mathcal{Z}_E ( \mu \setminus F) \right) = \kappa_E^{-1}( \mathcal{Z}_E ( \mu ) ) \setminus \left( \bigcup_{ e \in F } \kappa_E^{-1} ( \mathcal{Z}_E(\mu e))\right)
%\]
%Since $\mathcal{Z}_E ( \mu \setminus F) \neq \emptyset$, then $\mathcal{Z}_E ( \mu \setminus F ) = \mathcal{Z}_E ( \mu )$ or $r( \mu ) \notin E^0_\cycle$.  If $\mathcal{Z}_E ( \mu \setminus F ) = \mathcal{Z}_E ( \mu )$, then $\kappa_E^{-1} ( \mathcal{Z}_E(\mu) ) = \mathcal{Z}_{E_\curlyvee} ( \mu')$ for some $\mu' \in E^*_\curlyvee$.  Suppose $r( \mu ) \notin E^0_\cycle$.  Then $\mu \in E^*_{\curlyvee}$, $F$ is a finite subset of $r(\mu) E^1_{\curlyvee}$, and 
%\[
%\kappa_E^{-1} \left( \mathcal{Z} ( \mu \setminus F) \right) = \mathcal{Z}_{E_\curlyvee} ( \mu ) \setminus \left( \bigcup_{ e \in F }  \mathcal{Z}_{E_\curlyvee} (\mu e) \right) = \mathcal{Z}_{E_\curlyvee} ( \mu \setminus F ).
%\]
We have just shown that $\kappa_E$ is continuous.
 
Let $\mu \in E^*_{\curlyvee}$.  Then $\kappa_E ( \mathcal{Z}_{E_\curlyvee} ( \mu ) ) = \mathcal{Z}_E ( \mu )$.  Let $F$ be a finite subset of $r(\mu) E^1_{\curlyvee}$.  Then $\kappa_E ( \mathcal{Z}_{E_\curlyvee }( \mu e  ) ) = \mathcal{Z}_{E}( \mu e  ) $ for all $e\in F$, hence
\[
\kappa_E ( \mathcal{Z}_{E_\curlyvee} ( \mu \setminus F ) ) =  \mathcal{Z}_{E} ( \mu ) \setminus \left( \bigcup_{ e \in F }  \mathcal{Z}_{E}( \mu e  )  \right) = \mathcal Z_E(\mu\setminus F),
\]  
which is open.
%\[
%\kappa_E ( \mathcal{Z}_{E_\curlyvee} ( \mu \setminus G ) ) = \kappa_E ( \mathcal{Z}_{E_\curlyvee} ( \mu ) ) \setminus \left( \bigcup_{ e \in G } \kappa_E ( \mathcal{Z}_{E_\curlyvee }( \mu e  ) ) \right) =  \mathcal{Z}_{E} ( \mu ) \setminus \left( \bigcup_{ e \in G }  \mathcal{Z}_{E}( \mu e  )  \right),
%\]  
%where the first equal sign is true since $\kappa_E$ is a bijection.  
Hence, $\kappa_E$ is an open map.  Therefore, $\kappa_E$ is a homeomorphism.

Define $\ftn{m}{\partial E^{\geq 1}}{\N_0}$ by 
\[
m(y) = 
\begin{cases}
0, &\text{if $y \in \kappa_E ( E^0_\cycle )$} \\
1, &\text{otherwise.}
\end{cases}
\]
Since $E^0_\cycle \subseteq (E_\curlyvee)^0_\sink$, by Lemma~\ref{l:sinks-closed}, $E^0_\cycle$ is clopen in $\partial E_\curlyvee$.  Since $\kappa_E$ is a homeomorphism, $\kappa_E ( E^0_\cycle )$ is clopen in $\partial E$, so $m$ is continuous.
% if and only if $m \vert_{ \kappa_E ( E^0_\cycle ) }$ and $m \vert_{ \partial E^{\geq 1} \setminus \kappa_E ( E^0_\cycle ) }$ are continuous.  Since $m \vert_{ \kappa_E ( E^0_\cycle ) }$ and $m \vert_{ \partial E^{\geq 1} \setminus \kappa_E ( E^0_\cycle ) }$ are constant functions, both functions are continuous.  Thus, $m$ is continuous.

%
%Since $m^{-1}(\{0\}) = \bigcup_{v \in E^0_{\cycle}} \mathcal{Z}_E(v)$, we have that $m^{-1}(\{0\})$ is open.  Let $y \in m^{-1}(\{1\})$.  Let $\mu \in E^*$ such that $y \in \mathcal{Z}_E(\mu)$.  Since $m(y) =1$, we have that $\mu = e_1 e_2 \cdots e_n$ such that $s( e_1 ) \notin E^0_{\cycle}$.  Hence, for all $x \in \mathcal{Z}_E ( \mu )$, we have that $m(x) = 1$ which implies that $y \in \mathcal{Z}_E (\mu) \subseteq m^{-1}(\{1\})$.  So, $m$ is continuous.

A computation shows that 
\[
\sigma_E ( \kappa_E (x) ) = \kappa_E ( \sigma_{E_\curlyvee}(x)) \quad \text{and} \quad \sigma_{E_\curlyvee}^{m(y)} ( \kappa_E^{-1}(y)) = \kappa_E^{-1}( \sigma_{E}(y) ).
\]
Therefore, $\kappa_E$ is an orbit equivalence.  The last part of the proposition follows immediately from the construction of $\kappa_E$.
\end{proof}

The next lemma shows that we can adjust an orbit equivalence so that sinks are sent to sinks.  This will be used in the proof of Theorem~\ref{thm:main-result} to construct a diagonal preserving \stariso.  Note that if $\mu \in E^*$ such that $r( \mu )$ is a sink, then $\mu \in \partial E_\iso$.

\begin{lemma}\label{lem:adjusting-sinks}
Let $E$ be a graph and let $F$ be a subset of $E^{\geq 1}$ such that 
\begin{enumerate}
\item \label{eq1:adjusting-sinks} for all $\mu \in F$, $r(\mu) \in E^0_\sink$,

\item \label{eq2:adjusting-sinks} for all $\mu, \nu \in F$, $r( \mu ) = r( \nu )$ if and only if $\mu = \nu$, and 

\item \label{eq3:adjusting-sinks} $F$ is a closed subset of $\partial E$.
\end{enumerate}
Define $\ftn{\kappa}{\partial E}{\partial E}$ by 
\[
\kappa(x) = 
\begin{cases}
x &\text{if $x \notin F \cup r( F )$} \\
r(x) &\text{if $x \in F$} \\
\mu &\text{if $x = r( \mu )$ for some $\mu \in F$}.
\end{cases}
\]
Then $\kappa$ is an orbit equivalence.
\end{lemma}

\begin{proof}
We must show that $\kappa$ is a homeomorphism and there exist continuous functions $\ftn{l,m}{\partial E^{\geq 1}}{\N_0}$ and $\ftn{l',m'}{\partial E^{\geq 1 }}{\N_0}$ such that 
\[
\sigma_E^{m(x)} ( \kappa( \sigma_E(x) ) ) = \sigma_E^{l(x)} ( \kappa(x) ) \quad \text{and} \quad \sigma_E^{m'(x)} ( \kappa^{-1}( \sigma_E(x) ) ) = \sigma_E^{l'(x)} ( \kappa^{-1}(x) )
\]
for all $x \in \partial E^{\geq 1}$.  A computation shows that $\kappa \circ \kappa = \mathrm{id}$.  Hence, it is enough to show that $\kappa$ is continuous and there exist $\ftn{l,m}{\partial E^{\geq 1}}{\N_0}$ such that 
\[
\sigma_E^{m(x)} ( \kappa( \sigma_E(x) ) ) = \sigma_E^{l(x)} ( \kappa(x) ) 
\]
for all $x \in \partial E^{\geq 1}$.

Since $r(F) \subseteq E^0_\sink$, by Lemma~\ref{l:sinks-closed}, $r(F)$ is a clopen subset of $\partial E$.  Since $F \subseteq \partial E_\iso$, $F$ is open in $\partial E$.  By Assumption~\eqref{eq3:adjusting-sinks}, $F$ is closed in $\partial E$.
As $F$, $r(F)$ and thereby also $\partial E\setminus\left(F\cup r(F)\right)$ are clopen, it suffices to check for continuity on the three sets individually.
Since $F$ and $r(F)$ consist of isolated points and thereby carry the discrete subspace topology, $\kappa$ is automatically continuous on $F$ and $r(F)$.
Since $\kappa$ restricts to the identity on $\partial E\setminus\left(F\cup r(F)\right)$, $\kappa$ is also continuous on $\partial E\setminus\left(F\cup r(F)\right)$.
So $\kappa$ is continuous and hence a homeomorphism.

We now produce continuous functions $\ftn{l, m}{\partial E^{\geq 1 }}{\N_0}$ such that 
\[
\sigma_E^{m(x)} ( \kappa( \sigma_E(x) ) ) = \sigma_E^{l(x)} ( \kappa(x) )
\]
for all $x \in \partial E^{\geq 1}$.  Note that $F \subseteq E^*$ (paths of finite length).  For each $v \in r(F)$, we will denote the unique element in $F$ with range $v$ by $\mu_v$.  Note that if $x \in \sigma_E^{-1}(F)$, then $|x| \geq 2$ since $\sigma_E(x) \in F$ and $F \subseteq \partial E^{\geq 1}$.

Define $\ftn{l, m}{\partial E^{\geq 1}}{\N_0}$ by 
\[
l(x) = 
\begin{cases}
0 &\text{if $x \in F$,} \\
| x | &\text{if $x \in \sigma_E^{-1}(F)$,} \\
1 &\text{otherwise,}%if $x \in  \partial E^{\geq 1} \setminus (F \cup\sigma_E^{-1}(F) )$,} \\
\end{cases}
\]
and
\[
m(x) = 
\begin{cases}
| \mu_{r(x)} | &\text{if $x \in E^1 r(F)$,} \\
|x| - 1 &\text{if $x \in F \cap \partial E^{\geq 2}$,} \\
0 &\text{otherwise.}
\end{cases}
\]

We first show that $l$ is continuous.  Since $F$ is clopen in $\partial E^{\geq 1}$ and $\sigma_E$ is continuous, we have that %$F \cup clopen
$\sigma_E^{-1}(F)$ is a clopen subset of $\partial E^{\geq 1}$.  
Since $F$ and $\sigma_E^{-1}(F)$ consist of isolated points, $l$ is automatically continuous on $F$ and $\sigma_E^{-1}(F)$.
As $l$ is constant on $\partial E^{\geq 1}\setminus\left(F\cup \sigma_E^{-1}(F)\right)$, $l$ is also continuous on $\partial E^{\geq 1}\setminus\left(F\cup \sigma_E^{-1}(F)\right)$ and thereby on $\partial E^{\geq 1}$.
%Therefore, $l$ is continuous if and only if $l \vert_{ F \cup \sigma_E^{-1}(F) }$ and $l \vert_{ \partial E^{\geq 1} \setminus \left( F \cup \sigma_E^{-1}(F) \right) }$ are continuous.  Since $F \cup \sigma_E^{-1}(F) \subseteq \partial E_\iso$, the subspace topology of $F \cup \sigma_E^{-1}(F)$ is the discrete topology.  Thus, $l \vert_{ F \cup \sigma_E^{-1}(F) }$ is continuous.  Since $l \vert_{ \partial E^{\geq 1} \setminus \left( F \cup \sigma_E^{-1}(F) \right) }$ is a constant function, it is continuous.  We now can conclude that $l$ is continuous. 

For continuity of $m$, we first note that $F\cap\partial E^{\geq 2}$ and $E^1r(F)$ are both clopen in $\partial E^{\geq 1}$ as $F$ and $\partial E^{\geq 2}$ are, and as $E^1r(F)=\sigma_E^{-1}(F)\cap E^1$ with $\partial E^{\geq 1}\cap E^1$ clopen in $\partial E^{\geq 1}$.
Since $E^1r(F)$ and $F\cap\partial E^{\geq 2}$ consist of isolated points, they carry the discrete subspace topology, so $m$ is continuous on $E^1r(F)$ and $F\cap\partial E^{\geq 2}$.
As $m$ is constant on the complement $\partial E^{\geq 1}\setminus (E^1r(F)\cup(F\cap\partial E^{\geq 2}))$ we conclude that $m$ is continuous.
%For continuity of $m$, note that for $n \geq 1$, we have that $m^{-1}(\{n\}) \subseteq \partial E_{\iso}$ since the range of every element of $m^{-1} (\{n\})$ is a sink.  Thus, $m^{-1}(\{n\})$ is an open subset of $\partial E$.  Let $y \in m^{-1}(\{0\})$.  If $y \in \partial E_{\iso}$, then $\{y\}$ is an open subset of $\partial E$ containing $y$ such that $\{y\} \subseteq m^{-1}(\{0\})$.  Now assume that $y \notin \partial E_{\iso}$.  So in particular, $y \notin F$ and $|y| \geq 1$.  Set $y = e_1 \cdots$.  Then $r(e_1)$ is not a sink since $y$ is not an isolated point in $\partial E$.    Since $F$ is closed, $\mathcal{Z} ( e_1 ) \setminus  F$ is an open subset of $\partial E$ containing $y$.  For all $x \in \mathcal{Z} ( e_1 ) \setminus   F$, $x \notin E^1r(F)$ since $r(e_1) \notin E^0_{\sink }$.  Hence, $\mathcal{Z} ( e_1 ) \setminus   F \subseteq m^{-1}(\{0\})$.   So, $m$ is continuous. 

Let $x \in \partial E^{\geq 1}$.  Suppose $x \in E^1r( F )$.  Then 
\[
\sigma_E^{m(x)} ( \kappa( \sigma_E (x) ) ) = \sigma_E^{| \mu_{r(x)} |}  ( \kappa(r(x)) )  = \sigma_E^{| \mu_{r(x)} |}  ( \mu_{r(x)} ) = r(x)
\]
and 
\begin{align*}
\sigma_E^{l(x)} ( \kappa(x) )  
&= 
\begin{cases}
\kappa( x ) &\text{if $x \in F$} \\
\sigma_E( \kappa( x )) &\text{if $x \in E^1r(F) \setminus F$} 
\end{cases} \\
&= r(x).
\end{align*}

Suppose $x \in F \cap \partial E^{\geq 2}$.  
Then $\sigma_E(x)\notin F$ since $r(\sigma_E(x))=r(x)$ with $x\in F$, and $\sigma_E(x)\notin r(F)$ as $\sigma_E(x)\in\partial E^{\geq 1}$, hence $\kappa(\sigma_E(x))=\sigma_E(x)$.
%Write $x = e_1 e_2 \cdots e_{|x|}$.  Note that by Assumption~\eqref{eq2:adjusting-sinks} $e_2\cdots e_{|x|} \notin F$ since $r(x) = r(e_2\cdots e_{|x|})$ and $x \neq e_2 \cdots e_{|x|}$.  
Thus 
\[
\sigma_E^{m(x)} ( \kappa( \sigma_E (x) ) ) =  \sigma_E^{|x| - 1}( \sigma_E(x) )  = r( x)
\]
and 
\[
\sigma_E^{l(x)} ( \kappa(x) ) = \sigma_E^0(r(x)) = r(x). 
\]

Suppose $x \notin E^1 r( F )$ and $x \notin F \cap \partial E^{\geq 2}$.  Then 
$x\notin F$ and $\sigma_E(x)\notin r(F)$, so
\begin{align*}
\sigma_E^{m(x)} ( \kappa( \sigma_E (x) ) ) &= \kappa( \sigma_E(x) ) \\
						&= \begin{cases}
r( \sigma_E(x) )		&\text{if $x \in \sigma_E^{-1} ( F )$} \\
\sigma_E(x)			&\text{if $x \notin \sigma_E^{-1}(F)$} 
\end{cases}  \\
&= \begin{cases}
r( x )		&\text{if $x \in \sigma_E^{-1} ( F )$} \\
\sigma_E(x)			&\text{if $x \notin \sigma_E^{-1}(F)$,} 
\end{cases}
\end{align*}
and $\kappa(x)=x$, hence
%Recall that if $x \in \sigma_E^{-1} ( F )$, then $x \in \partial E^{\geq 2}$.  Thus, if $x \in \sigma_E^{-1}(F)$, then $x \notin F$ since $x \notin F \cap \partial E^{\geq 2}$.  Also, if $|x| = 1$, then $x \notin F$ since $x \notin E^1 r(F)$.  Therefore,    
\begin{align*}
\sigma_E^{l(x)} ( \kappa (x) ) ) &= 
\begin{cases}
\sigma_E^{|x|} ( x)		&\text{if $x \in \sigma_E^{-1} ( F )$} \\
\sigma_E( x )		&\text{if $x \notin \sigma_E^{-1}(F)$} 
\end{cases} \\
&= \begin{cases}
r( x )		&\text{if $x \in \sigma_E^{-1} ( F )$} \\
\sigma_E(x)			&\text{if $x \notin \sigma_E^{-1}(F)$.} 
\end{cases}
\end{align*}
We have just shown that $\ftn{l, m}{\partial E^{\geq 1 }}{\N_0}$ are continuous functions and
\[
\sigma_E^{m(x)} ( \kappa( \sigma_E(x) ) ) = \sigma_E^{l(x)} ( \kappa(x) )
\]
for all $x \in \partial E^{\geq 1}$.  We conclude that $\kappa$ is an orbit equivalence.
\end{proof}

%%%%%%%%%%%%%%%%%%%%%%%%%%%%%%%%%%%%%%%%%%%%%%%%%%%%%%%%%%
%%% SECTION: The Extended Weyl groupoid %%%%%%%%%%%%%%%%%%%%%%%%%%%%%%%%
%%%%%%%%%%%%%%%%%%%%%%%%%%%%%%%%%%%%%%%%%%%%%%%%%%%%%%%%%%

\section{The extended Weyl groupoid}

In \cite{arXiv:1410.2308}, the authors prove that a diagonal preserving \stariso between $C^*(E)$ and $C^*(F)$ implies that $E$ and $F$ are orbit equivalent.  In this section, we point out that their arguments even prove that the existence of a diagonal preserving \stariso between $C^*(E)$ and $C^*(F)$ implies the existence of an orbit equivalence between $E$ and $F$ preserving periodic points.  The arguments are actually contained in \cite[Section~4]{arXiv:1410.2308}.  For the convenience of the reader, we provide the arguments here.

First we need to recall the extended Weyl groupoid of $( C^*(E), \mathcal{D} (E) )$ as defined in \cite[Section~4]{arXiv:1410.2308}.

\begin{definition}
Let $E$ be a graph.  The \emph{normalizer} of $\mathcal{D} (E)$ is defined to be the set 
\[
 N( \mathcal{D} (E) ) = \setof{ n \in C^* (E) }{ \text{$ndn^* , n^* d n \in \mathcal{D}(E)$ for all $d \in \mathcal{D}(E)$}}. 
\]
\end{definition}

By \cite[Lemma~4.6]{MR2460017}, for all $n \in N( \mathcal{D} (E) )$, $nn^*$ and $n^*n$ are elements in $\mathcal{D}(E)$.  Therefore, we may define for $n \in N( \mathcal{D} (E)$, the sets 
\[
\mathrm{dom}(n) = \setof{ x \in \partial E }{ h_E (x) ( n^* n ) > 0 } 
\]
and 
\[
\mathrm{ran}(n) = \setof{ x \in \partial E }{ h_E (x) ( n n^* ) > 0 }. 
\]
By \cite[Proposition~4.7]{MR2460017}, for each $n \in N( \mathcal{D} (E) )$, there exists a unique homeomorphism $\ftn{\alpha_n}{\mathrm{dom}(n)}{\mathrm{ran}(n)}$ such that for all $d \in \mathcal{D}(E)$,
\[
h_E ( x ) ( n^* d n ) = h_E ( \alpha_n (x) ) (d) h_E (x)(n^*n).
\]

For each $x \in \partial E_\iso$, we let $p_x$ denote the unique element in $\mathcal{D} (E)$ satisfying $h_E (y) (p_x ) = 1$ if $y = x$ and zero otherwise, i.e., $p_x$ is the unique element in $\mathcal{D} (E)$ corresponding to the characteristic function $\chi_{ \{x\} } \in C_0( \partial E )$ under the canonical \starisos $\mathcal{D} (E) \cong C_0 ( \mathrm{Spec} ( \mathcal{D} (E) ) ) \cong C_0 ( \partial E )$.
By \cite[Lemma~4.3]{arXiv:1410.2308}, if $x \in \partial E_\iso$, then $p_x \mathcal{D} (E) p_x$ is either isomorphic to $\C$ (when $x$ is not eventually periodic) or $C( \mathbb{T})$ (when $x$ is eventually periodic).
%Recall that for each $x \in \partial E_\iso$, $p_x$ is the unique element in $\mathcal{D} (E)$ satisfying $h_E (y) (p_x ) = 1$ if $y = x$ and zero otherwise. 

 By \cite[Lemma~4.4]{arXiv:1410.2308}, for each $x \in \partial E_\iso$, $n_1, n_2 \in N( \mathcal{D} (E) )$ such that $x \in \mathrm{dom}(n_1) \cap \mathrm{dom}(n_2)$ and $\alpha_{n_1} (x) = \alpha_{n_2}(x)$,
\[
U_{(x, n_1, n_2 )} = \left( h_E(x)( n_1^* n_1 n_2^* n_2) \right)^{-1/2}p_x n_1^*n_2 p_x
\]
is a unitary in $p_x C^*(E) p_x$.  

We define an equivalence relation $\sim$ on $\setof{ (n, x ) }{ n \in N( \mathcal{D} (E) ), x \in \mathrm{dom}(n) }$ by $(n_1, x_1) \sim (n_2, x_2)$ if either 
\begin{enumerate}
\item $x_1 = x_2 \in \partial E_{\iso}$, $\alpha_{n_1}(x_1) = \alpha_{n_2}( x_2 )$, and $[ U_{( x_1, n_1, n_2 ) } ] = 0$ in $K_1 ( p_{x_1}C^*(E)p_{x_1} )$

\item $x_1 = x_2 \notin \partial E_{\iso}$ and there is an open set $V$ such that $x_1 \in V \subseteq \mathrm{dom} ( n_1) \cap \mathrm{dom} ( n_2)$ and $\alpha_{n_1} (y) = \alpha_{n_2} (y)$ for all $y \in V$.  
\end{enumerate}
It is shown in \cite[Proposition~4.6]{arXiv:1410.2308} that this relation is in fact an equivalence relation. 

Let $\mathcal{G}_{( C^*(E), \mathcal{D}(E) ) }$ be the collection of equivalence classes.  Define a partially defined product by 
\[
[ ( n_1 , x_1 )] \cdot [(n_2, x_2) ] = [ ( n_1 n_2, x_2 ) ] \quad \text{if $\alpha_{n_2}( x_2) = x_1$} 
\]
and undefined otherwise, define an inverse map by 
\[
[ ( n, x )  ]^{-1} = [ ( n^* , \alpha_n (x) ) ].
\] 
By \cite[Proposition~4.7 and Proposition~4.8]{arXiv:1410.2308}, $\mathcal{G}_{( C^*(E), \mathcal{D}(E) ) }$ is a groupoid and is a topological groupoid with the topology generated by 
\[
\setof{ \setof{[ (n,x) ] }{x \in \mathrm{dom}(n) } }{ n \in N( \mathcal{D}(E) ) }. 
\]
Moreover, by \cite[Proposition~4.8 and its proof]{arXiv:1410.2308}, the map $\phi_E$ from $\mathcal{G}_E$ to $\mathcal{G}_{ ( C^*(E), \mathcal{D}(E) ) }$ defined by 
\[
\phi_E ( ( x, k, y ) ) = [ ( s_\mu s_\nu^* , y ) ]
\]
where $x = \mu z$, $y = \nu z$, $k = |\mu | - | \nu |$ for some $\mu, \nu \in E^*$ and $z \in \partial E$, and $\setof{ p_v, s_e }{ v \in E^0, e \in E^1}$ be a Cuntz-Krieger $E$-family generating $C^*(E)$,  is an isomorphism.

\begin{proposition}[{cf.~\cite[Proposition~4.11]{arXiv:1410.2308}}]\label{prop:diagonalorbit}
Let $E$ and $F$ be graphs.  Suppose there exists a diagonal preserving \stariso $\ftn{\Psi}{C^* (E)}{C^* (F)}$. % such that $\Psi ( \mathcal{D} (E) ) = \mathcal{D} (F)$.  
Then there exists an isomorphism $\ftn{\beta}{\mathcal{G}_E}{\mathcal{G}_F}$ and a homeomorphism $\ftn{\kappa}{\partial E}{\partial F}$ such that $\beta ( ( \mu, 0 , \mu ) ) = ( \kappa ( \mu ), 0 , \kappa(\mu) )$ and for all $x \in \partial E_{\iso}$, $x$ is eventually periodic if and only if $\kappa (x)$ is eventually periodic.
\end{proposition}

\begin{proof}
Let $\setof{ p_v, s_e }{ v \in E^0, e \in E^1}$ be a Cuntz-Krieger $E$-family generating $C^*(E)$.  Since $\ftn{\Psi}{C^* (E)}{C^* (F)}$ is a \stariso such that $\Psi ( \mathcal{D} (E) ) = \mathcal{D} (F)$, there exists a homeomorphism $\ftn{\kappa}{\partial E }{\partial F}$ such that $h_E (x) ( f) = h_F ( \kappa (x) ) \Psi(f)$ for all $f \in \mathcal{D}(E)$ and the map $\ftn{\lambda}{\mathcal{G}_{ ( C^*(E), \mathcal{D}(E) ) }}{\mathcal{G}_{ ( C^* (F) , \mathcal{D}(E) ) }}$ given by 
\[
\lambda ( [( n, x ) ] ) = [ ( \Psi(x) , \kappa (x) ) ]
\]
is an isomorphism.  Now, $\ftn{\beta = \phi_F^{-1} \circ \lambda \circ \phi_E}{\mathcal{G}_E}{\mathcal{G}_F}$ is an isomorphism. 

We claim that $\beta ( ( \mu, 0 , \mu ) ) = ( ( \kappa (\mu ), 0 , \kappa ( \mu ) )$ for all $\mu \in \partial E$.  Let $\mu \in \partial E$.  Then 
\[
\phi_F (\beta ( ( \mu, 0 , \mu ) ) ) = \lambda ( [ ( s_\mu s_\mu^* , \mu ) ] ) = [ ( \Psi ( s_\mu s_\mu^*) , \kappa ( \mu ) ) ].
\]
Since $\phi_F$ is an isomorphism, there exists $( \nu, 0 , \nu ) \in G_F^{(0)}$ such that $\phi_F ( ( \nu, 0, \nu ) ) = [ ( \Psi ( s_\mu s_\mu^* ) , \kappa ( \mu ) ) ]$.  Since $\phi_F ( ( \nu, 0, \nu ) ) = [ ( s_\nu s_\nu^* , \nu ) ]$, we have that $\nu = \kappa( \mu )$.  Hence, $\phi_F^{-1} ( [ ( \Psi ( s_\mu s_\mu^* ) , \kappa ( \mu ) ) ] ) = ( \kappa( \mu ), 0 , \kappa ( \mu ) )$, thus proving the claim.

We will now show that for all $x \in \partial E_{\iso }$, $x$ is eventually periodic if and only if $\kappa(x)$ is eventually periodic.  To do this, we first show that for all $x \in \partial E_{\iso}$, then $\Psi ( p_x ) = p_{ \kappa (x) }$.  Let $x \in \partial E_{\iso}$.  Suppose $y \in \partial F$.  Then 
\begin{align*}
h_F (y) ( \Psi ( p_x) ) &= h_E ( \kappa^{-1}(y) ) (p_x)  \\
				&= 
				\begin{cases}
				1 &\text{if $\kappa^{-1}( y) = x$} \\
				0 &\text{otherwise}
				\end{cases} \\
				&= 
				\begin{cases}
				1 &\text{if $\kappa( x) = y$} \\
				0 &\text{otherwise}
				\end{cases} \\
				&= h_F ( y) ( p_{\kappa(x)} ).
\end{align*}
Therefore, by the uniqueness of $p_{\kappa(x)}$, we have that $\psi( p_x ) = p_{\kappa(x)}$.  Hence, proving the claim.  

Let $x \in \partial E_{\iso}$.  Then 
\[
p_{\kappa(x) } \mathcal{D}(F) p_{\kappa(x)} = \Psi ( p_x \mathcal{D} (E) p_x ),
\]
and hence $p_{\kappa(x) } \mathcal{D}(F) p_{\kappa(x)} \cong p_x \mathcal{D} (E) p_x $.  Therefore, $x$ is eventually periodic if and only if $p_x \mathcal{D}(E) p_x \cong C(\mathbb{T})$ if and only if $p_{\kappa(x) } \mathcal{D}(F) p_{\kappa(x)} \cong C( \mathbb{T} )$ if and only if $\kappa (x)$ is eventually periodic.
\end{proof}

\begin{theorem}\label{thm:diagonalorbit}
Let $E$ and $F$ be graphs.  Suppose there exists a diagonal preserving \stariso $\ftn{\Psi}{C^*(E)}{C^*(F)}$. % such that $\Psi ( \mathcal{D}(E) ) = \mathcal{D}(F)$.  
Then there exists an orbit equivalence between $E$ and $F$ preserving periodic points.
\end{theorem}

\begin{proof}
By Proposition~\ref{prop:diagonalorbit}, there exist an isomorphism $\ftn{\beta}{\mathcal{G}_E}{\mathcal{G}_F}$ and a homeomorphism $\ftn{\kappa}{\partial E}{\partial F}$ such that $\beta ( ( \mu, 0 , \mu ) ) = ( \kappa ( \mu ), 0 , \kappa(\mu) )$ and for all $x \in \partial E_{\iso}$, $x$ is eventually periodic if and only if $\kappa (x)$ is eventually periodic.  One can check that $\kappa \circ \mathcal{P}_E \circ \kappa^{-1} = \mathcal{P}_F$.  By Proposition~\ref{prop:arXiv:1410.2308:Prop3.4}, $\kappa$ is an orbit equivalence between $E$ and $F$ preserving periodic points.
\end{proof}

\section{Main result}

Let $E$ be a graph and let $S$ be a subset of $E^0_\sink$.  Define $E_{\curlywedge, S}$ to be the graph with vertices $E_{\curlywedge,S}^0 = E^0$ and edges
\[
E_{\curlywedge, S}^1 = E^1 \sqcup \setof{ e(v) }{ v \in S }
\]
where the range and source maps extends the range and source maps of $E$ respectively, and $r_{E_{\curlywedge, S} }( e(v) ) = s_{E_{\curlywedge, S} } ( e(v) ) = v$ for all $v \in S$.

\begin{proposition}\label{prop:iso-unplug-to-plug}
Let $E$ and $F$ be graphs, $S_1$ be a nonempty subset of $E^0_\sink$, $S_2$ be a nonempty subset of $F^0_\sink$.  Suppose there exist a bijection $\ftn{w}{ S_1 }{ S_2 }$ and a diagonal preserving \stariso $\ftn{\Phi}{ C^*(E) }{ C^*(F) }$ such that %$\Phi ( \mathcal{D} ( E ) ) = \mathcal{D} ( F )$ and 
$\Phi ( P_v ) = Q_{w(v)}$ for all $v \in S_1$, where $\setof{P_v, S_e}{ v \in E^0 , e \in E^1}$ is a Cuntz-Krieger $E$-family generating $C^*( E)$ and $\setof{Q_v, T_e}{ v \in F^0, e \in F }$ is a Cuntz-Krieger $F$-family generating $C^*( F )$.  Then there exists a diagonal preserving \stariso $\ftn{\Psi}{C^*(E_{\curlywedge, S_1}) }{ C^*( F_{\curlywedge, S_2} ) }$. %such that $\Psi ( \mathcal{D} ( E_{\curlywedge, S_1} ) ) = \mathcal{D} ( F_{\curlywedge, S_2} )$.
\end{proposition}

\begin{proof}
Let $\setof{p_v, s_e}{ v \in E^0_{\curlywedge, S_1}, e \in E^1_{\curlywedge, S_1} }$ be a Cuntz-Krieger $E_{\curlywedge, S_1}$-family generating $C^*( E_{\curlywedge, S_1} )$ and $\setof{q_v, t_e}{ v \in F^0_{\curlywedge, S_2}, e \in F^0_{\curlywedge, S_2} }$ be a Cuntz-Krieger $F_{\curlywedge, S_2}$-family generating $C^*( F_{\curlywedge, S_2} )$.  Clearly, %A computation shows that 
\[
\setof{ p_v, s_e }{ v \in E^0, e \in E^1 }
\]
is a Cuntz-Krieger $E$-family in $C^*( E_{\curlywedge, S_1} )$.   Therefore, there exists a \starhomo $\ftn{\Phi_1}{ C^*( E) }{ C^*( E_{\curlywedge, S_1} ) }$ such that $\Phi_1 ( P_v ) = p_v$ and $\Phi_1 ( S_e ) = s_e$ for all $v \in E^0$ and $e \in E^1$.  For all $\mu \in E^*$, $\Phi_1 ( S_\mu ) = s_\mu$.  Hence, for $\mu$ a vertex-simple cycle in $E$ with no exits, $\mu$ is a vertex-simple cycle in $E_{\curlywedge, S_1}$ with no exits, and so $\Phi_1 ( S_\mu ) = s_\mu$ is a partial unitary with spectrum equal to $\mathbb{T}$.  And since $\Phi_1 (P_v ) \neq 0$ for all $v \in E^0$, by \cite[Theorem~1.2]{MR1914564}, we have that $\Phi_1$ is injective.  Similarly, there exists an injective \starhomo $\ftn{\Phi_2}{ C^*( F) }{ C^*( F_{\curlywedge, S_2} ) }$ such that $\Phi_2 ( Q_v ) = q_v$ and $\Phi_2 ( T_e ) = t_e$ for all $v \in F^0$ and $e \in F^1$.

Set $\mathfrak A = \Phi_1 ( C^*(E))$ and $\mathfrak B = \Phi_2 ( C^*(F) )$.  Note that $\mathfrak A$ is the $C^*$-subalgebra of $C^*( E_{\curlywedge, S_1} )$ generated by $\setof{ p_v, s_e }{ v \in E^0, e \in E^1 }$ and $\mathfrak B$ is the $C^*$-subalgebra of $C^*( F_{\curlywedge, S_1} )$ generated by $\setof{ q_v, t_e }{ v \in F^0, e \in F^1 }$.  Moreover, 
\[
\Phi_1 ( \mathcal{D} (E) ) = \overline{\mathrm{span}} \setof{ s_\mu s_\mu^* }{ \mu \in E^*}
\]
which we denote by $\mathcal{D}(\mathfrak A)$ and 
\[
\Phi_2 ( \mathcal{D} (F) ) = \overline{\mathrm{span}} \setof{ t_\mu t_\mu^* }{ \mu \in F^*}
\] 
which we denote by $\mathcal{D}(\mathfrak B)$.  Therefore, $\Phi$ induces a \stariso $\ftn{ \widetilde{\Phi}}{\mathfrak A}{\mathfrak B}$ such that $\widetilde{\Phi} ( \mathcal{D} (\mathfrak A) ) = \mathcal{D} (\mathfrak B )$ and $\widetilde{\Phi} ( p_v ) = q_{w(v)}$ for all $v \in S_1$.

For $v \in E^0_{\curlywedge,S_1}$, set $\overline{p}_v = \widetilde{\Phi} ( p_v )$ and for $e \in E^1_{\curlywedge, S_1}$, set
\[
\overline{ s}_e = 
\begin{cases}
\widetilde{\Phi} ( s_e ) &\text{if $e \in E^1$} \\
t_{e( w(v) )} &\text{if $e = e(v)$ for some $v \in S_1$}. 
\end{cases}
\]
One can check that $\setof{ \overline{p}_v , \overline{s}_e }{ v \in E^0_{\curlywedge, S_1} , e \in E^1_{\curlywedge, S_1} }$ is a Cuntz-Krieger $E_{\curlywedge, S_1}$-family in $C^*( F_{\curlywedge, S_2 })$.  Hence, there exists a \starhomo $\ftn{\Psi}{ C^*( E_{\curlywedge, S_1} ) }{ C^* ( F_{\curlywedge, S_2} ) }$ such that $\Psi ( p_v ) = \overline{ p}_v$ and $\Psi ( s_e ) = \overline{s}_e$ for all $v \in E^0_{\curlywedge, S_1}$ and $e \in E^1_{\curlywedge, S_1}$.  We claim that $\Psi$ is a \stariso.

First we show that $\Psi$ is surjective.  Let $w \in F^0_{\curlywedge, S_2}$.  Since $F^0_{\curlywedge, S_2} = F^0$, we have that $w \in F^0$.  Thus, $q_{w}$ is in the image of $\widetilde{\Phi}$ and hence $q_w$ is in the image of $\Psi$.  Let $e \in F^1_{\curlywedge, S_2}$.  Suppose $e \in F^1$.  Then $t_e$ is in the image of $\widetilde{\Phi}$ which implies that $t_e$ is in the image of $\Psi$.  Suppose $e = e(z)$ for some $z \in S_2$.  Since $\ftn{w}{S_1}{S_2}$ is a bijection, $z = w(v)$ for some $v \in S_1$.  Hence, $\Psi ( s_{e(v)} ) = t_{ e(w(v) ) } = t_{e(z)} = t_e$.  Thus, $\Psi$ is surjective.

To show that $\Psi$ is injective, we will first show that for every vertex-simple cycle $\mu$ in $E_{\curlywedge, S_1}$ with no exits, $\Psi ( s_\mu )$ is a partial unitary with spectrum equal to $\mathbb{T}$.  Let $\mu$ be a vertex-simple cycle in $E_{\curlywedge, S_1}$ with no exits.  Note from the construction of $E_{\curlywedge, S_1}$, $\mu$ is either a vertex-simple cycle in $E$ with no exits or $\mu = e(v)$ for some $v \in S_1$.  Suppose $\mu$ is a vertex-simple cycle in $E$ with no exits.  Then $\Psi ( s_{\mu} ) = \widetilde{\Phi} ( s_{\mu} )$, and since $\widetilde{ \Phi }$ is a \stariso, $\Psi ( s_\mu ) = \widetilde{\Phi} ( s_\mu )$ is a partial unitary with spectrum equal to $\mathbb{T}$.  Suppose $\mu = e(v)$ for some $v \in S_1$.  Then $\Psi ( s_\mu ) = t_{ e(w(v)) }$ which is a partial unitary with spectrum equal to $\mathbb{T}$.  Since $\Psi ( p_v ) \neq 0$ for all $v \in E^0_{\curlywedge, S_1}$, by \cite[Theorem~1.2]{MR1914564} $\Psi$ is injective.

We have just shown that $\Psi$ is a \stariso.  We are left with showing that $\Psi ( \mathcal{D} ( E_{\curlywedge, S_1} ) ) = \mathcal{D} ( F_{\curlywedge, S_2} )$.  Noting that $s_{\mu e(v)^n } s_{\mu e(v)^n}^* = s_\mu s_\mu^*$ for all $v \in S_1$ and $\mu\in E^*$ with $r(\mu)=v$, we have that $\mathcal{D} ( E_{\curlywedge, S_1} ) = \mathcal{D}(\mathfrak A)$.  Similarly, $\mathcal{D} ( F_{\curlywedge, S_2} ) = \mathcal{D} (\mathfrak B)$.  It is now clear that $\Psi ( \mathcal{D} ( E_{\curlywedge, S_1} ) ) = \mathcal{D} ( F_{\curlywedge, S_2} )$ since $\widetilde{\Phi} ( \mathcal{D} (\mathfrak A) ) = \mathcal{D}(\mathfrak B)$. 
\end{proof}

To use the above result to prove that an orbit equivalence preserving periodic points implies diagonal preserving isomorphism, we must show that we may reduce the problem to graphs satisfying the property that all vertex-simple cycles with no exits are loops.  These are the loops that we will unplug, and then plug again.

\begin{proposition}\label{prop:reduction-to-loops}
Let $E$ be a graph.  Then there exists a graph $F$ such that each vertex-simple cycle in $F$ with no exits is a loop and there exists a diagonal preserving \stariso $\ftn{\Psi}{C^*(F)}{C^*(E)}$. %such that $\Psi ( \mathcal{D}(F) ) = \mathcal{D}(E)$.
\end{proposition}

\begin{proof}
For each cycle $\mu$, let $V_\mu$ be the vertices that support the cycle $\mu$.  
%Set $W \subseteq E^0$ such that for each $v \in W$, $v = s( \mu )$ for some vertex-simple cycle with no exits, and for all $v, v' \in W$, if $v$ and $v'$ are vertices on a vertex-simple cycle with no exits $\mu$, then $v = v'$.  Thus, for each $v \in W$, there exists a unique vertex-simple cycle with no exits $\mu_v$ such that $s( \mu_v ) = v$.  
%
%Set $\mathcal{F} = \setof{ \mu_v }{ v \in W }$.  
If $\mu$ and $\nu$ are vertex-simple cycles in $E$ with no exits, then $V_\mu=V_\nu$ if and only if $V_\mu\cap V_\nu\neq\emptyset$.  Define a relation $\approx$ on the vertex-simple cycles in $E$ with no exits by $\mu\approx\nu$ if $V_\mu=V_\nu$.  Clearly, $\approx$ is an equivalence relation, and we may pick a set $\mathcal F$ of representatives of the equivalence classes.
Then $\mathcal{F} \subseteq E^*$ such that 
\begin{enumerate}[(a)]
\item For each $\mu \in \mathcal{F}$, $\mu$ is a vertex-simple cycle with no exits;
\item For each $\mu, \nu \in \mathcal{F}$, $V_\mu \cap V_\nu \neq \emptyset$ if and only if $\mu = \nu$; and 
\item For each vertex-simple cycle $\mu$ in $E$ with no exits, there exists $\nu \in \mathcal{F}$ such that $V_\mu = V_\nu$.
\end{enumerate}
%Such a set exists as $\mu\approx\nu$ defined by $V_\mu=V_\nu$ is an equivalence relation on the set of vertex-simple cycles in $E$ with no exits, with $V_\mu=V_\nu$ if and only if $V_\mu\cap V_\nu\neq\emptyset$.
%Then $\mathcal F$ is a set of representatives of the equivalence classes.

Set $S = \setof{ e \in E^1 }{ \text{$s(e) = s( \nu )$ for some $\nu \in \mathcal{F}$} }$.  Define $F$ by $F^0 = E^0$, $F^1 = \left( E^1 \setminus S \right) \sqcup \setof{ \overline{ \nu } }{ \nu \in \mathcal{F} }$, and $r_F \vert_{E^1 \setminus S } = r_E \vert_{E^1 \setminus S }$, $s_F \vert_{E^1 \setminus S } = s_E \vert_{E^1 \setminus S }$, $r_F( \overline{\nu} ) = s_F( \overline{\nu} ) = s_E( \nu )$ for all $\nu \in \mathcal{F}$.  It is clear from the construction of $F$, that each vertex-simple cycle in $F$ with no exits is a loop.  In fact, the vertex-simple cycles with no exits are $\setof{ \overline{ \nu } }{ \nu \in \mathcal{F} }$.

Let $\setof{ p_v, s_e }{ v \in E^0, e \in E^1 }$ be a generating Cuntz-Krieger $E$-family in $C^*(E)$ , and  
let  $\setof{ q_v, t_e }{ v \in F^0 , e \in F^1 }$ be a generating Cuntz-Krieger $F$-family in $C^*(F)$.  
We now define a Cuntz-Krieger $F$-family in $C^*(E)$.  For each $v \in F^0$, set $Q_v = p_v$ and for each $e \in E^1 \setminus S$, set $T_e = s_e$.  For $\nu \in \mathcal{F}$, set $T_{\overline{\nu}} = s_\nu$.  A computation shows that $\setof{ Q_v , T_e }{ v \in F^0 , e \in F^1 }$ is a Cuntz-Krieger $F$-family in $C^*(E)$.  Hence, there exists a \starhomo $\ftn{\Psi}{C^*(F)}{C^*(E)}$ such that $\Psi ( q_v ) = Q_v$ and $\Psi ( t_e ) = T_e$ for all $v \in F^0$, $e \in F^1$.   

Since the only vertex-simple cycles of $F$ with no exits are the $\overline{\nu}$'s and $\Psi ( t_{\overline{\nu}} ) = s_\nu$ where $\nu$ is a vertex-simple cycle with no exits, $\Psi ( t_{\overline{\nu} })$ is a partial unitary with spectrum equal to $\mathbb{T}$.  Since $\Psi ( q_v ) = p_v \neq 0$ for all $v \in F^0$, by \cite[Theorem~1.2]{MR1914564}, $\Psi$ is injective.  

We now show that $\Psi$ is surjective.  
Clearly, $ p_v, s_e\in\im\Psi$, for all $v \in E^0$ and $e \in E^1\setminus S$.
%the $s_e$'s where the edges are elements of 
%\[
%\bigcup_{ \substack{ \nu \in \mathcal{F}  \\ | \nu | \geq 2}} s_E(\nu) E^1.
%\]
For each $\nu \in \mathcal{F}$, $s_E^{-1} ( s_E( \nu ) ) = \{ e_\nu \}$ for some $e_\nu\in E^1$ since $\nu$ is a vertex-simple cycle with no exits.  So $S=\{e_\nu\mid\nu\in\mathcal F\}$.
Let $\nu\in\mathcal F$.  If $|\nu|=1$ then $\nu=e_\nu$ so $s_{e_\nu}=T_{\overline\nu}\in\im\Psi$.
Assume $|\nu|\geq 2$.
%To show that $\Psi$ is surjective, it is enough to show that $e_\nu \in \Psi ( C^* (F) )$ for all $\nu \in \mathcal{F}$ with $| \nu | \geq 2$.  Let $\nu \in \mathcal{F}$ with $| \nu | \geq 2$.  
Then $\nu = e_\nu \mu$ where $\mu = e_1 \cdots e_{|\nu| - 1}$ with each $e_i$ not an element of $S$.  Hence,  $s_{e_i} = T_{e_i} \in\im\Psi$ for each $i$.  Since $\nu$ is a vertex-simple cycle with no exits, we have that $s_{e_i} s_{e_i}^* = p_{s_E(e_i) }$.  
We now have
\begin{align*}
s_{e_\nu} &= s_{e_\nu} p_{s_E(e_1)} = s_{ e_{\nu} } s_{e_1} s_{e_1}^* \\ 
		&= s_{ e_{\nu} } s_{e_1} p_{s_E ( e_2) } s_{e_1}^* = s_{ e_{\nu} } s_{e_1} s_{e_2} s_{e_2}^*s_{e_1}^* \\
		&\ \vdots \\
		&= s_{ e_{\nu} } s_{\mu} s_{\mu}^* \\
		&= s_\nu s_\mu^* \\
		&= \Psi ( t_{\overline{\nu}} t_\mu) \in \psi ( C^*(F) ).
\end{align*}
Therefore, $\Psi$ is surjective, and hence $\Psi$ is a \stariso.  

It is clear that $\Psi ( \mathcal{D} (F) ) \subseteq \mathcal{D} (E)$.  Let $\mu \in E^*$.  Note that $s_\mu s_\mu^* = s_\nu s_\nu^*$ for a path $\nu\in E^*$ where all edges of $\nu$ are not in $S$.  
Indeed, if $\mu=\nu e_1\cdots e_n$ with $e_1\in S$, $e_2,\ldots, e_n\in E^1$, and $\nu\in E^*$ where all edges of $\nu$ are not in $S$, then $s(e_i)E^1=\{e_i\}$ for all $i$ since $e_1\in S$.  So $p_{r(\nu)}=s_{e_1\cdots e_n}s_{e_1\cdots e_n}^*$, hence $s_\mu s_\mu^*=%s_\nu s_{e_1\cdots e_n}s_{e_1\cdots e_n}^*s_\nu^*=
s_\nu s_\nu^*$.
Therefore, %$\nu$ is an element of $F^*$ such that 
$\Psi ( t_\nu t_\nu^* ) = s_\nu s_\nu^* = s_\mu s_\mu^* $.  Hence, $\Psi ( \mathcal{D} (F) ) \supseteq \mathcal{D}(E)$, so
$\Psi$ is diagonal preserving. %$\Psi ( \mathcal{D} (F) ) = \mathcal{D}(E)$.  
\end{proof}

We provide an example to illustrate the construction in the proof of Proposition~\ref{prop:reduction-to-loops}.
Let $E$ denote the left-most graph below.
Then $E$ contains four vertex-simple cycles with no exits. Let $\mathcal F=\{\nu\}$ with $\nu=e_3e_4e_1e_2$. Then $S=\{e_3\}$ and $F$ is the right-most graph below.
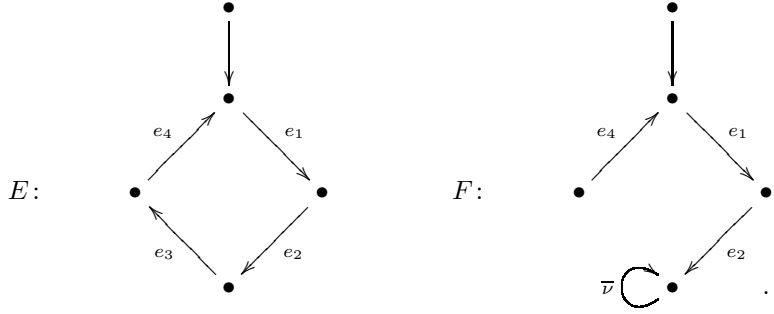
\[
 \xymatrix{ 
& & \bullet\ar[d] & \\
& & \bullet\ar[dr]^-{e_1} & \\
E \colon & \bullet\ar[ur]^-{e_4} & & \bullet\ar[dl]^-{e_2} \\
& & \bullet\ar[ul]^-{e_3} & 
 }
 \quad\quad\quad\quad
 \xymatrix{ 
&  & \bullet\ar[d] & \\
&  & \bullet\ar[dr]^-{e_1} & \\
F \colon & \bullet\ar[ur]^-{e_4} & & \bullet\ar[dl]^-{e_2} \\
& & \bullet \ar@(dl,ul)[]^-{\overline\nu} & .
 }
\]
%\[
% \xymatrix{ 
%& & \bullet\ar[d] & \\
%& & \bullet\ar[dr]^-{e_1} & \\
%E \colon & \bullet\ar[ur]^-{e_4} & & \bullet\ar[dl]^-{e_2} \\
%& & \bullet\ar[ul]^-{e_3} & .
% }
%\]
%\[
% \xymatrix{ 
%& & \bullet\ar[d] & \\
%& & \bullet\ar[dr]^-{e_1} & \\
%F \colon & \bullet\ar[ur]^-{e_4} & & \bullet\ar[dl]^-{e_2} \\
%& & \bullet \ar@(dl,ul)[]^-{\overline\nu} & .
% }
%\]

We are now ready to prove our main result.

\begin{theorem}\label{thm:main-result}
Let $E$ and $E'$ be graphs.  Then the following are equivalent.
\begin{enumerate}
\item \label{equivalence1}  There exists a diagonal preserving \stariso $\ftn{\Psi}{C^*(E)}{C^*(E')}$. %such that $\Psi ( \mathcal{D} (E) ) = \mathcal{D}(E')$.

\item \label{equivalence2} $\mathcal{G}_E$ and $\mathcal{G}_{E'}$ are isomorphic.

\item \label{equivalence3} There exists a homeomorphism $\ftn{\kappa}{\partial E}{\partial E'}$ such that $\kappa \circ \mathcal{P}_E \circ \kappa^{-1} = \mathcal{P}_{E'}$ and for each $x \in \partial E_{\iso}$, $x$ is eventually periodic if and only if $\kappa(x)$ is eventually periodic.

\item \label{equivalence4} There exists an orbit equivalence from $E$ to $E'$ preserving periodic points.
\end{enumerate}
\end{theorem}

\begin{proof}
The equivalence of \eqref{equivalence1} and \eqref{equivalence2} is \cite[Theorem~5.1 (1)$\iff$(2)]{arXiv:1410.2308}.  The equivalence of \eqref{equivalence3} and \eqref{equivalence4} follows from Proposition~\ref{prop:arXiv:1410.2308:Prop3.4}.  We are left to showing that (\ref{equivalence1}) and  (\ref{equivalence4}) are equivalent. 

By Theorem~\ref{thm:diagonalorbit}, we get that \eqref{equivalence1} implies \eqref{equivalence4}.  We now prove \eqref{equivalence4} implies \eqref{equivalence1}.

We will first show that we may assume that all vertex-simple cycles in $E$ and $E'$ with no exits are loops.  Indeed, by Proposition~\ref{prop:reduction-to-loops}, there are graphs $E_1$ and $E'_1$ such that all vertex-simple cycles in $E_1$ and $E'_1$ with no exits are loops, and diagonal preserving \starisos from $C^*(E)$ to $C^*( E_1)$ and from $C^*(E')$ to $C^*( E'_1)$.
By Theorem~\ref{thm:diagonalorbit}, there exists an orbit equivalence from $E$ to $E_1$ preserving periodic points and there exists an orbit equivalence from $E'$ to $E'_1$ preserving periodic points.  By Proposition~\ref{prop:orbitequivalence-equivalencerelation}, there exists an orbit equivalence from $E_1$ to $E'_1$ preserving periodic points
if and only if there is an orbit equivalence from $E$ to $E'$ preserving periodic points.
And clearly, there is a diagonal preserving \stariso from $C^*(E_1)$ to $C^*(E_1')$ if and only if there is one from $C^*(E)$ to $C^*(E')$.  This establishes the claim.
 %Thus, if there exists a \stariso $\ftn{\Phi}{ C^* ( E_1 ) }{ C^*(E'_1)}$ such that $\Phi(\mathcal{D}( E_1 ) )= \mathcal{D} ( E'_1)$, then $\ftn{ \Psi = \Phi_2^{-1} \circ \Phi \circ \Phi_1 }{ C^*(E) } { C^*(E') }$ is a \stariso such that $\Psi ( \mathcal{D} ( E ) ) = \mathcal{D} ( E')$.  Hence, we may assume that every vertex-simple cycle in $E$ and $E'$ with no exits are loops.

Assume all vertex-simple cycles in $E$ and $E'$ with no exits are loops, and that there exists an orbit equivalence $\beta$ from $E$ to $E'$ preserving periodic points.  
Let $\ftn{\kappa_E }{\partial E_\curlyvee}{\partial E}$ and $\ftn{\kappa_{E'}}{\partial E'_\curlyvee}{\partial E'}$ be the orbit equivalences provided in Proposition~\ref{prop:oe-unplug}.  Then by Propostion~\ref{prop:orbitequivalence-equivalencerelation}, $\ftn{\lambda = \kappa_{E'}^{-1} \circ \beta \circ \kappa_E }{\partial E_\curlyvee }{ \partial E'_\curlyvee}$ is an orbit equivalence.  Let $V =\setof{ v \in E^0_{\cycle }}{ \lambda ( v ) \in (\partial E'_\curlyvee)^{\geq 1} }$.  Set $F = \lambda (V)$.  By Lemma~\ref{l:sinks-closed}, $V$ is closed in $\partial E_\curlyvee$.  Since $\lambda$ is a homeomorphism, we have that $F$ is closed in $\partial E'_\curlyvee$.  Hence, $F$ satisfies \eqref{eq3:adjusting-sinks} in Lemma~\ref{lem:adjusting-sinks}. 

Let $v \in E^0_\cycle$.  Then, $\kappa_E ( v ) = e_v^\infty$.  Since $\beta$ is an orbit equivalence preserving periodic points, there exist $w \in (E')^0_\cycle$ and $\mu \in (E')^*$ with no edges equal to $e_w$ and $r_{E'}( \mu ) = w$ such that $\beta ( e_v^\infty ) = \mu e_w^\infty$.  So, in particular, $\mu \in (E'_\curlyvee)^*$ with $r_{E'}( \mu ) \in (E')^0_\cycle$.  Therefore, $\lambda( v ) = \kappa_{E'}^{-1} ( \beta ( \kappa_E ( v ) ) ) = \kappa_{E'}^{-1}( \mu e_w^\infty )= \mu$.  Hence, $r_{E'_\curlyvee} ( \lambda ( E^0_\cycle ) ) \subseteq (E')^0_\cycle$.
In particular, $F\subseteq (E'_\curlyvee)^{\geq 1}$ with $r_{E'_\curlyvee} (F ) \subseteq (E_\curlyvee')^0_\sink$, so $F$ satisfies \eqref{eq1:adjusting-sinks} in Lemma~\ref{lem:adjusting-sinks}. 
A similar argument using $\lambda^{-1}$ shows that $r_{E_\curlyvee} ( \lambda^{-1} ( (E')^0_\cycle ) ) \subseteq E^0_\cycle$.  
%Hence, $r_{E'_\curlyvee} ( \lambda ( E^0_\cycle ) )= (E')^0_\cycle$.  So in particular, since $(E')^0_\cycle \subseteq (E_\curlyvee')^0_\sink$, $r_{E'} ( \lambda ( V ) ) \in (E_\curlyvee')^0_\sink$.  So, $r_{E'_\curlyvee} ( F ) \subseteq (E'_\curlyvee)^0_\sink$ which implies that $F$ satisfies \eqref{eq1:adjusting-sinks} in Lemma~\ref{lem:adjusting-sinks}.      

Let $v\in (E')^0_\cycle$.  Then $\lambda^{-1}(v)=\mu\in E_\curlyvee^*$ with $r_E(\mu)=w\in E^0_\cycle$ and $\beta(\mu e_w^\infty)=e_v^\infty$.
By~\cite[Lemma~3.5]{arXiv:1410.2308}, there exist $n,m\in\N_0$ such that 
\[ \sigma_{E'}^n(\beta(\sigma_E^{|\mu|}(\mu e_w^\infty)))=\sigma_{E'}^n(\beta(\mu e_w^\infty)) .\]
As $\sigma_{E'}^n(\beta(\mu e_w^\infty))=e_v^\infty$,  $\beta(e_w^\infty)=\nu e_v^\infty$ for some $\nu\in (E')^*$ with $r_{E'}(\nu)=v$ and no edges in $\nu$ equal to $e_v$.
Hence $\lambda(w)=\nu$ with $r_{E'_\curlyvee}(\nu)=v$, so
\[ r_{E'_\curlyvee}(\lambda(r_{E_\curlyvee}(\lambda^{-1}(v))))
= r_{E'_\curlyvee}(\lambda(r_{E_\curlyvee}(\mu)))
= r_{E'_\curlyvee}(\lambda(w))
=v . \]
Applying this to $r_{E_\curlyvee} ( \lambda^{-1} ( (E')^0_\cycle ) ) \subseteq E^0_\cycle$ we see that $ (E')^0_\cycle \subseteq r_{E'_\curlyvee} ( \lambda ( E^0_\cycle ) )$ and conclude that $ (E')^0_\cycle = r_{E'_\curlyvee} ( \lambda ( E^0_\cycle ) )$.

We now show that $F$ satisfies \eqref{eq2:adjusting-sinks} in Lemma~\ref{lem:adjusting-sinks}.
Let $v_1,v_2\in E^0_\cycle$ and assume that $r_{E'_\curlyvee}( \lambda(v_1) ) = r_{E'_\curlyvee} ( \lambda(v_2) )$.  We will show that $\lambda (v_1) = \lambda (v_2)$.  %Choose $e_1, e_2 \in E^1_\cycle$ such that $s_E ( e_i ) = v_i$.  
Then $\beta ( e_{v_1}^\infty ) = \mu_1 e_{w_1}^\infty$ and $\beta ( e_{v_2}^\infty ) = \mu_2 e_{w_2}^\infty$ where $w_i \in (E')^0_\cycle$, $\mu_i \in (E')^*$, no edges in $\mu_i$ are equal to $e_{w_i}$ and $r_{E'}( \mu_i ) = w_i$.  So, $\lambda( v_i ) = \mu_i$, which implies that $r_{E'_\curlyvee} ( \lambda (v_i) ) = w_i$.  
%Since $r_{E'_\curlyvee}( \lambda(v_1) ) = r_{E'_\curlyvee} ( \lambda(v_2) )$, we have that $s_{E'}( f_1 ) = s_{E'}( f_2)$.  
So $w_1=w_2$.
%So, $f_1 = f_2$.  
Since $\beta$ is an orbit equivalence, by \cite[Lemma~3.5]{arXiv:1410.2308}, there exist $n_1,m_1, n_2, m_2 \in \N_0$ such that
\[
\sigma_E^{n_i} ( \beta^{-1} ( \sigma_F^{|\mu_i|} ( \mu_i e_{w_1}^\infty ))) = \sigma_E^{m_i} ( \beta^{-1} ( \mu_i e_{w_1}^\infty ))
\] 
for all $i\in\{1,2\}$.
%\[
%\left( \sigma_E^n \circ \beta^{-1} \circ \sigma_F^{|\mu_1|} \right) ( \mu_1 f^\infty ) = \left(\sigma_E^m \circ \beta^{-1} \right)( \mu_1 f^\infty )
%\]
%and 
%\[
%\left( \sigma_E^k \circ \beta^{-1} \circ \sigma_F^{|\mu_2|} \right) ( \mu_2 f^\infty ) = \left(\sigma_E^l \circ \beta^{-1} \right)( \mu_2 f^\infty ).
%\]
Thus,
\[ e_{v_1}^\infty = \sigma_E^{n_2}(e_{v_1}^\infty) = \sigma_E^{n_1+n_2}(\beta^{-1}(e_{w_1}^\infty)) =  \sigma_E^{n_1}(e_{v_2}^\infty) = e_{v_2}^\infty .\]
This implies that $v_1=v_2$.
%Thus,
%\begin{align*}
%e_1^\infty &= \sigma_E^{m+k} ( e_1^\infty ) = \sigma_E^k \left( \left(\sigma_E^m \circ \beta^{-1} \right)( \mu_1 f^\infty ) \right) = \sigma_E^k \left( \left( \sigma_E^n \circ \beta^{-1} \circ \sigma_F^{|\mu_1|} \right) ( \mu_1 f^\infty ) \right)  \\
%		&= \sigma_E^k \left( \left( \sigma_E^n \circ \beta^{-1} \right) ( f^\infty ) \right) = \sigma_E^n \left( \left( \sigma_E^k \circ \beta^{-1} \circ \sigma_F^{|\mu_2|} \right) ( \mu_2 f^\infty ) \right) \\
%		 &= \sigma_E^n \left( \left(\sigma_E^l \circ \beta^{-1} \right)( \mu_2 f^\infty ) \right) = \sigma_E^{n+l} ( e_2^\infty ) = e_2^\infty,
%\end{align*}
%which implies that $e_1 = e_2$.  
%Hence, 
%\begin{align*}
%\lambda ( v_1 ) &= \left( \kappa_{E'}^{-1} \circ \beta \circ \kappa_E\right)( v_1 ) = \left( \kappa_{E'}^{-1} \circ \beta \right) ( e_1^\infty ) =\left( \kappa_{E'}^{-1} \circ \beta \right) ( e_2^\infty )  \\
%			&= \left( \kappa_{E'}^{-1} \circ \beta \circ \kappa_E\right)( v_2 ) = \lambda( v_2).
%\end{align*}
We have just shown that $F$ satisfies \eqref{eq2:adjusting-sinks} in Lemma~\ref{lem:adjusting-sinks}.

Let $\ftn{\kappa}{\partial E'_\curlyvee}{ \partial E'_\curlyvee}$ be the orbit equivalence given in Lemma~\ref{lem:adjusting-sinks} for the graph $E'_\curlyvee$ and the set $F$.  Then $\gamma = \kappa \circ \lambda$ is an orbit equivalence from $\partial E_{\curlyvee}$ to $\partial E'_{\curlyvee}$.  
Let $v\in E^0_\cycle$.  If $\lambda(v)\in F$ then $\kappa(\lambda(v))=r_{E'_\curlyvee}(\lambda(v))$.
If $\lambda(v)\notin F$ then $\lambda(v)=r_{E'_\curlyvee}(\lambda(v))\notin r_{E'_\curlyvee}(F)$ as we saw above that $r_{E'_\curlyvee}(\lambda(v))=r_{E'_\curlyvee}(\lambda(w))$ implies $v=w$ for all $w\in E^0_\cycle$.
Either way, $\gamma(v) = \kappa(\lambda(v))=r_{E'_\curlyvee}(\lambda(v))$.
Since $r_{E'_\curlyvee} ( \lambda ( E^0_\cycle ) )= (E')^0_\cycle$,
% and $\kappa$ sends $\lambda (v)$ for $v \in E^0_\cycle$ to $r_{E'_{\curlyvee}}( \lambda(v) )$, 
 we have that $\gamma ( E^0_{\cycle} ) = (E')^0_{\cycle}\subseteq (E'_\curlyvee)^0_\sink$.  Since $E_\curlyvee$ and $E'_\curlyvee$ are graphs satisfying Condition~(L), by Theorem~\ref{thm:arXiv:1410.2308}, there exists an isomorphism $\ftn{\phi}{\mathcal{G}_{E_\curlyvee}}{ \mathcal{G}_{E'_\curlyvee}}$ such that $\phi \vert_{\partial E_\curlyvee} = \gamma$.  
 By~\cite[Proposition~2.2]{arXiv:1410.2308},
this isomorphism of groupoids induces a \stariso $\ftn{\Phi}{C^*( E_\curlyvee )}{C^*( E'_\curlyvee )}$ such that $\Phi ( \mathcal{D} ( E_\curlyvee) ) = \mathcal{D} ( E'_\curlyvee)$ and $\Phi ( p_v ) = q_{\gamma(v)}$ for all $v \in E^0_\cycle$, where $\setof{ p_v, s_e }{ v\in E^0_\curlyvee, e \in E^1_\curlyvee}$ is a Cuntz-Krieger $E_\curlyvee$-family generating $C^*( E_\curlyvee)$ and $\setof{ q_v, t_e }{ v \in (E'_\curlyvee)^0, e \in (E'_\curlyvee)^1}$ is a Cuntz-Krieger $E'_\curlyvee$-family generating $C^*( E'_\curlyvee)$.  Since $(E_\curlyvee)_{\curlywedge, E^0_\cycle } \cong E$ and $(E'_\curlyvee)_{\curlywedge, (E')^0_\cycle } \cong E'$, Proposition~\ref{prop:iso-unplug-to-plug} implies that there exists a diagonal preserving \stariso $\ftn{\Psi}{ C^*(E) }{ C^*(E')}$. %such that $\Psi ( \mathcal{D} ( E ) ) = \mathcal{D} ( E' )$.  
Hence \eqref{equivalence4} implies \eqref{equivalence1}.  
\end{proof}

We will denote the $C^*$-algebra of compact operators on $\ell^2(\N)$ by $\mathbb{K}$ and the maximal abelian subalgebra of $\mathbb{K}$ consisting of diagonal operators by $\mathcal{C}$.  For a commutative ring $R$ with identity, we write $\mathsf{M}_\infty(R)$ for the ring of finitely supported, countably infinite square matrices over $R$ and $\mathsf D_\infty (R)$ for the abelian subring of $\mathsf{M}_\infty (R)$ consisting of diagonal matrices.  

We write $( C^*(E), \mathcal{D} (E) ) \cong ( C^*(F) , \mathcal{D}(F) )$ if there exists a diagonal preserving \stariso $\Psi$ from $C^*(E)$ to $C^*(F)$, %such that $\Psi ( \mathcal{D} (E) ) = \mathcal{D}(F)$
 and write $( C^*(E) \otimes \mathbb{K} , \mathcal{D} (E) \otimes \mathcal{C} ) \cong ( C^*(F) \otimes \mathbb{K}, \mathcal{D}(F) \otimes \mathcal{C} )$ if there exists a \stariso from $\Psi$ from $C^*(E) \otimes \mathbb{K}$ to $C^*(F) \otimes \mathbb{K}$ such that $\Psi ( \mathcal{D} (E) \otimes \mathcal{C}  ) = \mathcal{D}(F) \otimes \mathcal{C}$.  

Let $R$ be a commutative ring with identity.  If there is a ring \stariso $\Psi$ from $L_R(E)$ to $L_R(F)$ such that $\Psi ( \mathcal{D}_R(E) ) = \mathcal{D}_R(F)$, then we write 
\[
( L_R(E), \mathcal{D}_R(E) ) \cong ( L_R(F), \mathcal{D}_R(F) ).
\]  
Similarly, if there is a ring \stariso $\Psi$ from $L_R(E) \otimes \mathsf{M}_\infty(R)$ to $L_R(F) \otimes \mathsf{M}_\infty(R)$ such that $\Psi ( \mathcal{D}_R(E) \otimes \mathsf D_\infty (R) ) = \mathcal{D}_R(F) \otimes \mathsf D_\infty (R)$, then we write 
\[
( L_R(E) \otimes \mathsf{M}_\infty(R), \mathcal{D}_R(E) \otimes \mathsf D_\infty (R) ) \cong (L_R(F) \otimes \mathsf{M}_\infty(R), \mathcal{D}_R(F) \otimes \mathsf D_\infty (R) ).
\]

Let $\mathcal{R}$ be the full equivalence relation on $\N \times \N$.  We can regard $\mathcal{R}$ as a discrete principal groupoid with unit space $\N$.
\begin{corollary}\label{bigcor}
Let $E$ and $F$ be graphs, and let $R$ be a commutative integral domain with~1.
The following are equivalent:
\begin{enumerate}
\item\label{it:CE sdcong CF} $( C^*(E)
    \otimes \mathbb{K}, \mathcal{D}(E) \otimes \mathcal{C} ) \cong ( C^*(F) \otimes \mathbb{K}, \mathcal{D}(F) \otimes \mathcal{C} )$;
\item\label{it:LE sdcong LF} $( L_R(E) \otimes \mathsf M_\infty(R)  , \mathcal{D}_R(E) \otimes \mathsf D_\infty(R) ) \cong ( L_R(F) \otimes \mathsf M_\infty(R), \mathcal{D}_R(F) \otimes \mathsf D_\infty(R) )$;
\item\label{it:CSE dcong CSF} $( C^*(SE), \mathcal{D}(SE) ) \cong ( C^*(SF) , \mathcal{D}(SF) )$; 
\item\label{it:LSE dcong LSF} $( L_R(SE), \mathcal{D}_R(SE) ) \cong (L_R(SF) , \mathcal{D}_R(SF) )$;
\item\label{it:E sge F} $\mathcal{G}_E \times \mathcal{R} \cong \mathcal{G}_F \times \mathcal{R}$;
\item\label{it:SE ge SF} $\mathcal{G}_{SE} \cong \mathcal{G}_{SF}$;
\item\label{it:SE oe SF} There exists an orbit equivalence from $SE$ to $SF$ preserving periodic points.
\end{enumerate}
\end{corollary}

\begin{proof}
By \cite[Theorem~4.2]{arXiv:1602.02602}, (\ref{it:CE sdcong CF}) through (\ref{it:SE ge SF}) are equivalent.  (\ref{it:SE ge SF}) $\iff$ (\ref{it:SE oe SF}) follows from Theorem~\ref{thm:main-result} for the graphs $SE$ and $SF$.
\end{proof}

We end by noting that the results above combine with \cite{preprint:ERRS} to completely resolve the relationship between orbit and flow equivalence for countable shift spaces. 

\begin{corollary}
Let $E$ and $F$ be finite graphs with no sinks and sources such that 
the edge shift spaces $\mathsf{X}_E$ and $\mathsf{X}_F$ are countable sets.
\begin{enumerate}
\item\label{it:direct} 
 If $E$ and $F$ are orbit equivalent, then  $\mathsf{X}_E$ and $\mathsf{X}_F$ are flow equivalent. 
\item\label{it:stable} $SE$ and $SF$ are orbit equivalent if and only if $\mathsf X_E$ and $\mathsf X_F$ are flow equivalent.
\end{enumerate}
\end{corollary}

\begin{proof}
Suppose $E$ and $F$ are orbit equivalent.  Since $E$ and $F$ have no sinks, there exists an orbit equivalence from $E$ to $F$ preserving periodic points.  Hence, by Theorem~\ref{thm:main-result}, $C^*(E) \cong C^*(F)$. One easily sees by contradiction that the asserted countability translates to the condition that every vertex of $E$ and $F$ either supports exactly one return path or does not support a return path.    Hence, \cite[Theorem~7.1 (3)$\implies$(4)]{preprint:ERRS} applies, and thus the shift spaces $\mathsf{X}_E$ and $\mathsf{X}_F$ are flow equivalent, proving \eqref{it:direct}.

For \eqref{it:stable}, we note that if  $SE$ and $SF$ are orbit equivalent then by Corollary \ref{bigcor}, $C^*(E) \otimes \mathbb{K} \cong C^*(F) \otimes \mathbb{K}$, and the forward implication follows as above. In the other direction, suppose $\mathsf X_E$ and $\mathsf X_F$ are flow equivalent.  By \cite[Lemma 5.1]{preprint:ERRS}, then $E$ and $F$ are Move equivalent.  By \cite[Corollary 4.8]{arXiv:1602.02602}, there exists a diagonal preserving isomorphism from $C^*(E) \otimes \mathbb K$ to $C^*(F) \otimes \mathbb K$.  Hence, by Corollary \ref{bigcor}, $SE$ and $SF$ are orbit equivalent.  
\end{proof}

%%%%%%%%%%%%%%%%%%%%%%%%%%%%%%%%%%%%%%%%%%%%%%%%%%%%%%%%%%
%%% SECTION: Acknowledgement %%%%%%%%%%%%%%%%%%%%%%%%%%%%%%%%
%%%%%%%%%%%%%%%%%%%%%%%%%%%%%%%%%%%%%%%%%%%%%%%%%%%%%%%%%%

\section{Acknowledgements}

This work was partially supported by the Danish National Research Foundation through the Centre for Symmetry and Deformation (DNRF92), by VILLUM FONDEN through the network for Experimental Mathematics in Number Theory, Operator Algebras, and Topology, and by a grant from the Simons Foundation (\# 279369 to Efren Ruiz).

This work was completed while all three authors were
attending the research program \emph{Classification of operator algebras: complexity,
rigidity, and dynamics} at the Mittag-Leffler Institute, January--April 2016. We thank the institute and its staff for the excellent work conditions provided.

The authors thank Aidan Sims for many helpful discussions.  
%%%%%%%%%%%%%%%%%%%%%%%%%%%%%%%%%%%%%%%%%%%%%%%%%%%%%%%%%%
%%% SECTION: Bibliography %%%%%%%%%%%%%%%%%%%%%%%%%%%%%%%%
%%%%%%%%%%%%%%%%%%%%%%%%%%%%%%%%%%%%%%%%%%%%%%%%%%%%%%%%%%

\providecommand{\bysame}{\leavevmode\hbox to3em{\hrulefill}\thinspace}
\providecommand{\MR}{\relax\ifhmode\unskip\space\fi MR }
\providecommand{\MRhref}[2]{%
  \href{http://www.ams.org/mathscinet-getitem?mr=#1}{#2}
}
\providecommand{\href}[2]{#2}


\begin{thebibliography}{ERRS16}

\bibitem[ADR00]{MR1799683}
C.~Anantharaman-Delaroche and J.~Renault, \emph{Amenable groupoids},
  Monographies de L'Enseignement Math\'ematique [Monographs of L'Enseignement
  Math\'ematique], vol.~36, L'Enseignement Math\'ematique, Geneva, 2000, With a
  foreword by Georges Skandalis and Appendix B by E. Germain.

\bibitem[BCW14]{arXiv:1410.2308}
Nathan Brownlowe, Toke~Meier Carlsen, and Michael~F. Whittaker, \emph{Graph
  algebras and orbit equivalence}, to appear in Ergod. Theory and Dynam. Systems, \url{doi:10.1017/etds.2015.52}.

\bibitem[CRS16]{arXiv:1602.02602}
Toke~Meier Carlsen, Efren Ruiz, and Aidan Sims, \emph{Equivalence and stable
  isomorphism of groupoids, and diagonal-preserving stable isomorphisms of
  graph {$C^*$}-algebras and {L}eavitt path algebras}, eprint arXiv:1602.02602
  (2016).

\bibitem[CW16]{cw}
Toke~M. Carsen and Marius~L. Winger, \emph{Orbit equivalence of graphs and
  isomorphism of graph groupoids}, private communication, 2016.

\bibitem[ERRS16]{preprint:ERRS}
S{\o}ren Eilers, Gunnar Restorff, Efren Ruiz, and Adam P.~W. S{\o}rensen,
  \emph{Geometric classification of graph {$C^*$}-algebras over finite graphs},
  eprint arXiv:1604.05439, 2016.

\bibitem[Mat13]{matsu}
Kengo Matsumoto, \emph{Classification of {Cuntz-Krieger} algebras by orbit
  equivalence of topological {Markov} shifts}, Proc. Amer. Math. Soc.
  \textbf{141} (2013), 2329--2342.

\bibitem[MM14]{MR3276420}
Kengo Matsumoto and Hiroki Matui, \emph{Continuous orbit equivalence of
  topological {M}arkov shifts and {C}untz-{K}rieger algebras}, Kyoto J. Math.
  \textbf{54} (2014), no.~4, 863--877, \href
  {http://dx.doi.org/10.1215/21562261-2801849}
  {\path{doi:10.1215/21562261-2801849}}.

\bibitem[Ren08]{MR2460017}
Jean Renault, \emph{Cartan subalgebras in {$C^*$}-algebras}, Irish Math. Soc.
  Bull. (2008), no.~61, 29--63.

\bibitem[R{\o}r95]{MR1340839}
Mikael R{\o}rdam, \emph{Classification of {C}untz-{K}rieger algebras},
  $K$-Theory \textbf{9} (1995), no.~1, 31--58, \href
  {http://dx.doi.org/10.1007/BF00965458} {\path{doi:10.1007/BF00965458}}.

\bibitem[Szy02]{MR1914564}
Wojciech Szyma{\'n}ski, \emph{General {C}untz-{K}rieger uniqueness theorem},
  Internat. J. Math. \textbf{13} (2002), no.~5, 549--555, \href
  {http://dx.doi.org/10.1142/S0129167X0200137X}
  {\path{doi:10.1142/S0129167X0200137X}}.

\bibitem[Web14]{MR3119197}
Samuel B.~G. Webster, \emph{The path space of a directed graph}, Proc. Amer.
  Math. Soc. \textbf{142} (2014), no.~1, 213--225, \href
  {http://dx.doi.org/10.1090/S0002-9939-2013-11755-7}
  {\path{doi:10.1090/S0002-9939-2013-11755-7}}.

\bibitem[Yee07]{MR2301938}
Trent Yeend, \emph{Groupoid models for the {$C^*$}-algebras of topological
  higher-rank graphs}, J. Operator Theory \textbf{57} (2007), no.~1, 95--120.

\end{thebibliography}
\end{document}